\documentclass[12pt,a4paper,oneside]{amsart}
\usepackage{amsmath,amssymb}
\usepackage{mathrsfs}
\usepackage{anysize}
\usepackage{url}
\usepackage[leqno]{amsmath}
\usepackage{enumitem}
\usepackage{tikz-cd}
\usetikzlibrary{arrows}
\usetikzlibrary{patterns}
\usepackage{bbding}
\usepackage{datetime}
\usepackage{pgfplots}
\usepackage{caption}
\usepackage{tikz}
\usepackage{accents}
\usetikzlibrary{calc}
\usetikzlibrary{snakes}
\usetikzlibrary{arrows}
\usetikzlibrary{backgrounds}
\usepackage{3dplot}
\tdplotsetmaincoords{50}{135}
\usepackage[left=2cm, right=2cm, top=2cm]{geometry} 
\usepackage{appendix}

\newtheorem{theorem}{Theorem}[section]
\newtheorem{lemma}[theorem]{Lemma}
\newtheorem{corollary}[theorem]{Corollary}
\newtheorem{proposition}[theorem]{Proposition}
\newtheorem{definition}[theorem]{Definition}

\usetikzlibrary{decorations.markings}
\tikzset{negated/.style={
        decoration={markings,
            mark= at position 0.5 with {
                \node[transform shape] (tempnode) {${\scriptstyle\setminus} $};
            }
        },
        postaction={decorate}
    }
}

\tikzset{degil/.style={
            decoration={markings,
            mark= at position 0.5 with {
                  \node[transform shape] (tempnode) {$\backslash$};
                  }
              },
              postaction={decorate}
}
}

\medskip
\title[A switched server system]{A switched server system   semiconjugate \\[0.05in] to a minimal Interval exchange }


\begin{document}

\maketitle

\centerline{\scshape Filipe Fernandes}

{\footnotesize
 \centerline{Departamento de Matem\'atica, Universidade Federal de S\~ao Carlos}
 \centerline {13565-905, S\~ao Carlos - SP, Brazil}
   \centerline{filipefernandes@dm.ufscar.br }
   
   \medskip

\centerline{\scshape\normalsize Benito Pires}

{\footnotesize
 \centerline{Departamento de Computa\c c\~ao e Matem\'atica, Faculdade de Filosofia, Ci\^encias e Letras}
 \centerline {Universidade de S\~ao Paulo, 14040-901, Ribeir\~ao Preto - SP, Brazil}
   \centerline{benito@usp.br} }
   
   \begin{abstract} Switched server systems are mathematical models of manufacturing, traffic and queueing systems that have being studied since the early 1990s. In particular, it is known that typically the dynamics of such systems is asymptotically periodic: each orbit of the system converges to one of its finitely many limit cycles. In this article, we provide an explicit example of a switched server system with exotic behavior: each orbit of the system converges to the same Cantor attractor. To accomplish this goal, we bring together recent advances in the understanding of the topological dynamics of piecewise contractions and interval exchange transformations with flips. The ultimate result is a switched server system whose Poincar\'e map is semiconjugate to a minimal and uniquely ergodic interval exchange transformation with flips.
\end{abstract}

\section{Introduction}

Certain aspects of manufacturing, traffic or queueing systems are captured by the mathematical model named \textit{switched server system}, which was introduced by Chase, Serrano and Ramadge in \cite[Section II.B, p. 72]{CSR1993}. It is a  continuous-time system discretely controlled via a switched state-feedback, also referred to as a hybrid dynamical system  (see \cite{AA2001}). It can also be considered a   \textit{pseudo-billiard} (see \cite{BB2004}).  In this article, we provide an example of a switched server system with atypical non-trivial dynamics. Our approach benefits from recent advances in the understanding of the topological dynamics of piecewise contractions (see \cite{BP2018}).       

The switched server system we consider here  consists of $3$ buffers (tanks) numbered $1$, $2$, $3$, and a server. It is very convenient to think of each buffer $i$ as a tank partially filled in with a fluid (work). At each time $t\ge 0$, a fluid is delivered to each tank $i$ at the constant rate $\rho_i=\frac13$ ($i=1,2,3$) and is removed from a selected tank $i\in\{1,2,3\}$ by the server  at the constant rate $\rho=1$. The volume of fluid in the tank $i$ at the time $t$ is denoted by $v_i(t)$. When the tank $i$ is emptied by the server at  the time $t$, the server changes its location to the tank $j\neq i$ with the largest scaled volume $d_{ij} v_j(t)$, where $\{d_{ij}:1\le i,j\le 3,i\neq j\}$ are the parameters of the system. We assume that $\sum_{i=1}^3 v_i(0)=1$. Since the system is closed ($\rho_1+\rho_2+\rho_3=\rho$), we have that  $\sum_{i=1}^3 v_i(t)=1$ for every $t\ge 0$. Hence, the state
$\mathbf{v}(t)=(v_1(t),v_2(t),v_3(t))$ of the system at the time $t$ is a probability vector and the phase space is the set $\Delta=\{\mathbf{v}=(v_1,v_2,v_3): v_i\ge 0,\forall i\,\,\textrm{and}\,\, v_1+v_2+v_3=1\}$. Let $l(t)$ denote the position of the server at the time $t$. We assume that $t\mapsto l(t)$ is right-continuous.
Figure \ref{thesss}.(a) shows a switched server system with the server located at  the position $l=1$.  

The trajectory $t\in [0,\infty)\mapsto \mathbf{v}(t)\in\Delta$ describes the position of a particle that moves with constant velocity inside $\Delta$ and changes its velocity when the particle hits the boundary $\partial\Delta$ according to a non-specular reflection. Hence, the system is a
\textit{pseudo-billiard} (see \cite{BB2004}). The times  $0\le t_1<t_2<t_3\ldots$ at which any of the tanks is empty are called the {\it switching times}. At the initial time $t=0$, the server is supposed to be connected to a non-empty tank.
Notice that $\mathbf{v}(t)\in \partial \Delta$ (the boundary of the phase space) if and only if $t\in \{t_1,t_2,\ldots\}$ (i.e. if $t$ is a switching time). In other words, at the switching times, the pseudo billiard trajectory hits the boundary $\partial\Delta$. By sampling the system at the switching times, we obtain a map $F:\partial\Delta\to\partial\Delta$ called the \textit{Poincar\'e map induced by the switched server system} (see Figure \ref{thesss}.(b)). The \textit{frequency} with which the server is connected to the tank $i$ is defined by
$$ \textit{freq}\,(i)=\lim_{n\to\infty} \frac{1}{n}\#{\{1\le k\le n: l(t_k)=i\}}.
$$

The dynamics of a switched server system with parameters $\{d_{ij}>0:1\le i,j\le 3, i\neq j\}$ depends only on the proportionality between pairs of parameters. More specifically, switched server systems sharing the same ratios 
$d_{13}/d_{12}$, $d_{21}/d_{23}$ and $d_{32}/d_{31}$ have the same dynamics. In this way,
we assume that if $(d_1,d_2,d_3)$ is a vector with positive entries, then the system parameters $d_{ij}$ are chosen according to the following conditions
 \begin{equation}\label{dij}
 \dfrac{d_{13}}{d_{12}}=d_1,\qquad \dfrac{d_{21}}{d_{23}}=d_2,\qquad \dfrac{d_{32}}{d_{31}}=d_3.
 \end{equation}
 By \cite[Theorem 1.4]{NPR2018}, we have that for Lebesgue almost every vector $(d_1,d_2,d_3)$ with positive entries, any switched server system with parameters $d_{ij}$ satisfying $(\ref{dij})$ is structurally stable and admits finitely many limit cycles that attract all the orbits. The same  result was obtained in  \cite[Theorem 4.1]{CSR1993} under the additional restrictions: $d_{21}=d_{31}$, $d_{12}=d_{32}$ and $d_{13}=d_{23}$. Figure \ref{thesss}.(b) shows the case in which $d_1=d_2=d_3=1$ and $d_{ij}=1$ for all $i\neq j$. In this case, $\left\{\left(0,\frac23,\frac13\right),\left(\frac13,0,\frac23\right),\left(\frac23,\frac13,0\right)\right\}$ is a limit cycle of the system.
 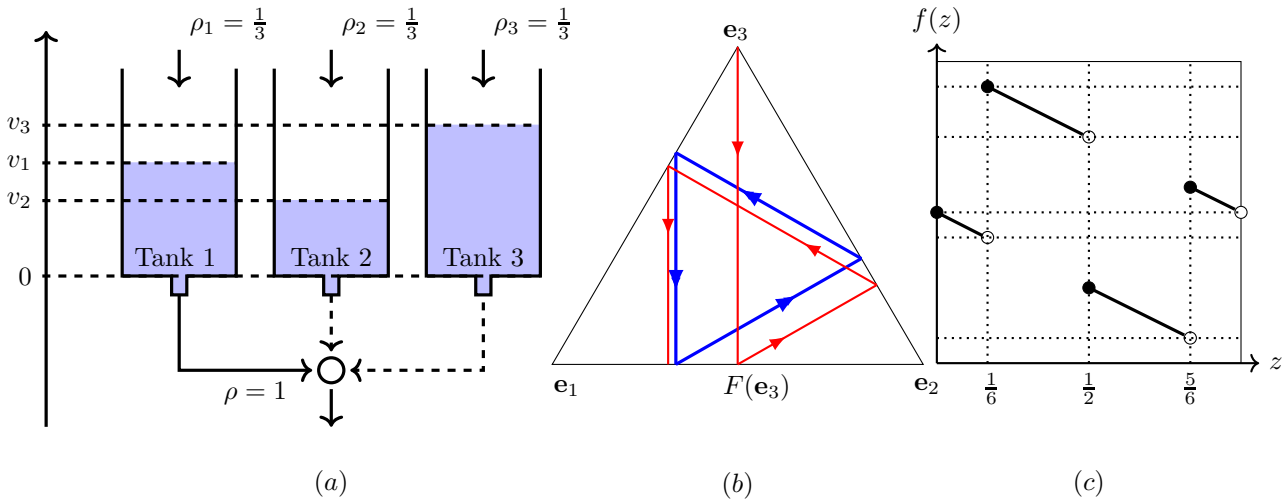
\begin{figure}[!htb]\vspace{-0.6cm}
   \begin{minipage}{0.3\textwidth}
      \centering
       \vspace*{1cm}
     \begin{tikzpicture}[scale=1]
\begin{scope}[shift={(0,0)}, scale=0.5]
    \draw[very thick] (0,6)--(0,0.5)--(1.3,0.5)--(1.3,0)--(1.7,0)--(1.7,0.5)--(3,0.5)--(3,6);
    \begin{pgfonlayer}{background}
        \filldraw[blue!25] (0,3.5)--(3,3.5)--(3,0.5)--(1.7,0.5)--(1.7,0)--(1.3,0)--(1.3,0.5)--(0,0.5)--cycle;
     \end{pgfonlayer}
      \draw[very thick,->] (1.5,6.5)--(1.5,5.5)node[pos=0,anchor=south west]{$\rho_1=\frac13$};
        \node at (1.4,1){Tank 1};
\draw[very thick,->] (1.5,0)--(1.5,-2)--(5,-2);
    \draw[very thick] (4,6)--(4,0.5)--(5.3,0.5)--(5.3,0)--(5.7,0)--(5.7,0.5)--(6,0.5)--(7,0.5)--(7,6);
    \begin{pgfonlayer}{background}
    \filldraw[blue!25] (4,2.5)--(4,0.5)--(5.3,0.5)--(5.3,0)--(5.7,0)--(5.7,0.5)--(6,0.5)--(7,0.5)--(7,2.5)--cycle;
     \end{pgfonlayer}
      \draw[very thick,->] (5.5,6.5)--(5.5,5.5)node[pos=0,anchor=south west]{$\rho_2=\frac13$};
      \node at (5.5,1){Tank 2};
       \draw[very thick,dashed,->] (5.5,0)--(5.5,-0.5)--(5.5,-1.5)node[pos=0,anchor=south west]{};
    \draw[very thick] (8,6)--(8,0.5)--(9.3,0.5)--(9.3,0)--(9.7,0)--(9.7,0.5)--(11,0.5)--(11,6);
    \begin{pgfonlayer}{background}
    \filldraw[blue!25] (8,4.5)--(8,0.5)--(9.3,0.5)--(9.3,0)--(9.7,0)--(9.7,0.5)--(11,0.5)--(11,4.5)--cycle;
     \end{pgfonlayer}
      \draw[very thick,->] (9.5,6.5)--(9.5,5.5)node[pos=0,anchor=south west]{$\rho_3=\frac13$};
      \node at (9.5,1){Tank 3};
       \draw[very thick,dashed,->] (9.5,-0.2)--(9.5,-2)--(6,-2)node[pos=0,anchor=south west]{};
       \draw[very thick,->] (5.5,-2.5)--(5.5,-3.5) node[xshift=-1cm, yshift=0.45cm]{$\rho=1$};
      
      
     \node [draw,circle, minimum width=0.2 cm,very thick](B) at (5.5,-2){}; 
     \draw[very thick, ->](-2,-3.5)--(-2,7);
      \draw[very thick,dashed](-2.1,3.5)--(3,3.5) node[pos=0,left]{$v_1$};
      \draw[very thick,dashed](-2.1,2.5)--(7,2.5) node[pos=0,left]{$v_2$};
       \draw[very thick,dashed](-2.1,4.5)--(11,4.5) node[pos=0,left]{$v_3$};
        \draw[very thick,dashed](-2.1,0.5)--(11,0.5) node[pos=0,left]{$0$};
        \draw[very thick,dashed] node at (5.5,-5) {$(a)$};

\end{scope}
\end{tikzpicture}
   \end{minipage}\hfill
   \begin {minipage}[m]{0.27\textwidth}
     \centering
      \vspace*{1.25cm}

\begin{tikzpicture}[tdplot_main_coords, scale=3.45]
\begin{scope}[shift={(0,1)}]
   \draw [,fill opacity=0.5]
          (1,0,0) -- (0,1,0) -- (0,0,1) -- cycle;
   \draw (1,0.2) node[left] {$\mathbf{e}_1$};
   \draw (0.1,1,-0.05) node[right] {$\mathbf{e}_2$};
   \draw (0.03,-0.1,1.03) node[right] {$\mathbf{e}_3$};
   \draw (0.03,-0.1,-0.75) node[right] {$F(\mathbf{e}_3)$};
 
\begin{scope}[blue, very thick,decoration={
    markings,
    mark=at position 0.65 with {\arrow[blue,scale=1.2]{latex}}}
    ] 
\draw[postaction={decorate}] ( 0, 0.666666666666667, 0.333333333333333 ) -- ( 0.333333333333333, 0, 0.666666666666667 ); 
\draw[postaction={decorate}] ( 0.333333333333333, 0, 0.666666666666667 ) -- ( 0.666666666666667, 0.333333333333333, 0 ); 
\draw[postaction={decorate}] ( 0.666666666666667, 0.333333333333333, 0 ) -- ( 0, 0.666666666666667, 0.333333333333333 ); 
 \end{scope}
\begin{scope}[blue, thin,decoration={
    markings,
    mark=at position 0.65 with {\arrow[blue,scale=1.2]{latex}}}
    ] 

\draw[postaction={decorate}] ( 0.333333333333333, 0, 0.666666666666667 ) -- ( 0.666666666666667, 0.333333333333333, 0 ); 
\draw[postaction={decorate}] ( 0.666666666666667, 0.333333333333333, 0 ) -- ( 0, 0.666666666666667, 0.333333333333333 ); 
\draw[postaction={decorate}] ( 0, 0.666666666666667, 0.333333333333333 ) -- ( 0.333333333333333, 0, 0.666666666666667 ); 
\end{scope}
\begin{scope}[red ,thick, decoration={
    markings,
    mark=at position 0.35 with {\arrow[red,scale=1.2]{latex}}}
    ] 

\draw[postaction={decorate}] ( 0, 0, 1 ) -- ( 0.5, 0.5, 0 ); 
\draw[postaction={decorate}] ( 0.5, 0.5, 0 ) -- ( 0, 0.75, 0.25 ); 
\draw[postaction={decorate}] ( 0, 0.75, 0.25 ) -- ( 0.375, 0, 0.625 ); 
\draw[postaction={decorate}] ( 0.375, 0, 0.625 ) -- ( 0.6875, 0.3125, 0 ); 
\end{scope}
\draw[very thick,dashed] node at (0,0,-1.2) {$(b)$};

    ] 

    
 \end{scope}  
 \end{tikzpicture}
   \end{minipage}
    \begin{minipage}{0.3\textwidth}
     \centering
     \vspace*{1cm}
\begin{tikzpicture}[scale=0.8]



\draw [  thick, ->] (0,0) -- (5.3,0) node [right] {\footnotesize $z$};
\draw [  thick, ->] (0,0) -- (0,5.3) node [above] {\footnotesize $f(z)$};
	
			

\draw (0,0)--(5,0)--(5,5)--(0,5);
			
\draw[fill=black] (0,5/2) circle (0.1);
\draw[fill=white] (5/6, 25/12) circle (0.1);
\draw[very thick] (0,5/2)--(5/6-0.08, 25/12+0.03) ;

\draw[fill=black] (5/6,55/12) circle (0.1);
\draw[fill=white] (5/2, 15/4) circle (0.1);
\draw[very thick] (5/6,55/12)--(5/2-0.1, 15/4+0.04) ;

\draw[fill=black] (5/2,5/4) circle (0.1);
\draw[fill=white] (25/6, 5/12) circle (0.1);
\draw[very thick] (5/2,5/4)--(25/6-0.1, 5/12+0.05) ;

\draw[fill=black] (25/6,35/12) circle (0.1);
\draw[fill=white] (5, 5/2) circle (0.1);
\draw[very thick] (25/6,35/12)--(5-0.1, 5/2+0.05) ;

\draw [  thick, dotted] (5/6,0)--(5/6,5);
\draw [  thick, dotted] (5/2,0)--(5/2,5);
\draw [  thick, dotted] (25/6,0)--(25/6,5);
\draw [  thick, dotted] (0,5/2)--(5,5/2);
\draw [  thick, dotted] (0,5/12)--(5,5/12);
\draw [  thick, dotted] (0,15/4)--(5,15/4);
\draw [  thick, dotted] (0,25/12)--(5,25/12);
\draw [  thick, dotted] (0,55/12)--(5,55/12);

\draw (0.6,0) node[below right] {$\frac16$};
\draw (5/2,0) node[below] {$\frac12$};
\draw (25/6,0) node[below] {$\frac56$};
\draw[very thick,dashed] node at (2.5,-2) {$(c)$};
 
\end{tikzpicture}
   \end{minipage}\hfill
   \caption{The switched server system, the pseudo billiard and the Poincar\'e map}\label{thesss}
\end{figure}

 In this article, we are interested in constructing switched server systems with complex dynamics, i.e., with no periodic orbit and therefore with no limit cycle. In the light of  what was discussed in the previous paragraph, it necessary to search for the appropriate parameters in a  Lebesgue negligible set of parameters  $(d_1,d_2,d_3)$. Moreover, the example we provide  presents stochastic regularity in the sense that it is possible to compute the frequency with which the server is connected to the tank $i$ at the switching times.

The strategy we use to tackle the problem is the following. The dynamics of a switched server system is completely determined by the Poincar\'e map $F:\partial\Delta\to\partial\Delta$ induced by the system on the boundary $\partial\Delta$ of the phase space. The Poincar\'e map $F$ is topologically conjugate to the piecewise smooth interval map
$f:[0,1]\to [0,1]$ defined by
$f=\varphi^{-1}\circ F\circ \varphi$, where $\varphi:[0,1]\to\partial\Delta$ denotes the  anticlockwise arc-length parametrization of $\partial\Delta$ with $\varphi(0)=\mathbf{e}_2=(0,1,0)$. Conversely, the following lemma is provided in this article:
\begin{lemma}\label{lmint}
Given $d_1,d_2,d_3>0$, let $f_{d_1,d_2,d_3}:[0,1]\to [0,1]$ be the map defined by
\begin{equation}\label{fd1d2d3}
f_{d_1,d_2,d_3}(z)=
\begin{cases} -\dfrac12 z + \dfrac12 & \textrm{if} \,\, z\in [z_0,z_1)\\[0.5em]
-\dfrac12 z + 1 & \textrm{if} \,\, z\in [z_1,z_2)\\[0.5em]
-\dfrac12 z + \dfrac12 & \textrm{if} \,\, z\in [z_2,z_3) \\[0.5em]
-\dfrac12 z + 1 & \textrm{if} \,\, z\in [z_3,z_4] \\[0.5em]
\end{cases},
\end{equation}
where
\begin{equation}\label{y123} 
z_0=0,\quad 
z_1=\dfrac{1}{3(1+d_1)},\quad z_2=\dfrac{1}{3(1+d_2)}+\dfrac{1}{3},\quad
z_3=\dfrac{1}{3(1+d_3)}+\dfrac{2}{3}, \quad z_4=1.
\end{equation}
Then the Poincar\'e map $F:\partial\Delta\to\partial\Delta$ of
 any switched server system with parameters $d_{ij}$ satisfying $(\ref{dij})$ is topologically conjugate to
  $f_{d_1,d_2,d_3}$. 
\end{lemma}
In Figure \ref{thesss}.(c), the map $f=f_{1,1,1}$ is plotted considering $d_1=d_2=d_3=1$. In general,  for any  $d_1,d_2,d_3>0$, the map $f_{d_1,d_2,d_3}$ is a piecewise $\lambda$-affine contraction, where $\lambda=\frac12$ (see \cite{NPR2018}). We say that an infinite word $i_0i_1\ldots$ over the alphabet $\mathcal{A}=\{1,2,3,4\}$ is a \textit{symbolic itinerary} or \textit{natural coding} of $f=f_{d_1,d_2,d_3}$ if there exists $z\in [0,1]$ such that, for each $k\ge 0$,
$$f^k(z)\in\begin{cases} [z_{i_k-1},z_{i_k}) &\textrm{if} \quad i_k<4 \\ [z_3,z_4] & \textrm{if} \quad i_k=4
 \end{cases}.$$
 
  The problem we want to solve translates into the following question.\\

\noindent \textbf{(Q)} Does the family of piecewise contractions $\{f_{d_1,d_2,d_3}: d_1>0, d_2>0,d_3>0\}$ contains
a map having no ultimately periodic symbolic itinerary (and therefore no periodic orbit and no limit cycle) ? \\

On the one hand, as already mentioned, recent advances (see \cite{NPR2014,NPR2018}) in the understanding of the topological dynamics of piecewise contractions show that generically piecewise contractions have finitely many limit cycles that attracts all orbits. Hence, an affirmative answer to $(\textbf{Q})$ is very unlikely. On the other hand, as it was shown very recently (see \cite[Theorem 2.2]{BP2018}), there exist piecewise $\frac12$-affine contractions with only one gap having no periodic orbit and no ultimately periodic symbolic itinerary. In order to adapt the proof of \cite[Theorem 2.2]{BP2018} to our framework, it is necessary to find an isometric model for $f_{d_1,d_2,d_3}$, that is, a minimal and uniquely ergodic interval exchange transformation (IET) $T$ with $4$ flips and $3$ discontinuities $0<x_1<x_2<x_3$ satisfying
$T(x_2)<T(0)<T(x_3)<T(x_1)$ (see Section \ref{Softher}). This step is very hard to accomplish because Nogueira proved in \cite{AN1989} that generically IETs with flips are not minimal. Surprisingly, as we show in this article, (\textbf{Q}) has an affirmative answer.

The use of interval exchange transformations as isometric models of complex dynamics is quite standard. 
 Lots of piecewise smooth aperiodic interval maps are topologically semiconjugate to  IETs (see \cite{RCCG1997, MC2002, CG1986, BP2016, BP2018}). Moreover, IETs are the simplest discontinuous interval maps preserving Lebesgue measure (see \cite{MK1975}). 

\section{Statement of the results}\label{Softher}

Throughout this article, let $P$ and $Q$ be the integer matrices defined by
$$P=\left(\begin{array}{cccc}
3 & 3 & 5 & 4 \\
1 & 2 & 3 & 3 \\
1 & 1 & 2 & 1 \\
2 & 3 & 5 & 5
\end{array} \right),\quad
\quad 
Q=\left(\begin{array}{cccc}
1 & 0 & 0 & 0 \\
0 & 1 & 1 & 0 \\
0 & 0 & 1 & 1 \\
0 & 1 & 0 & 0
\end{array} \right).
$$\smallskip
Let $\boldsymbol{\nu}$ be the probability eigenvector with positive entries associated with the Perron-Frobenius eigenvalue $\eta$ of $P$. Let $\boldsymbol{\lambda}=(\lambda_1,\lambda_2,\lambda_3,\lambda_4)$ be the vector defined by 
$
\boldsymbol{\lambda}=Q\boldsymbol{\nu}
$
whose norm is $\vert\boldsymbol{\lambda}\vert=\lambda_1+\lambda_2+\lambda_3+\lambda_4>1$. 
Consider the partition of the interval $[0,\vert\boldsymbol{\lambda}\vert]$:
 $$ I_1=[0,\lambda_1),\,\, I_2=[\lambda_1,\lambda_1+\lambda_2), \,\, I_3=[\lambda_1+\lambda_2,\lambda_1+\lambda_2+\lambda_3), \,\, I_4=[\lambda_1+\lambda_2+\lambda_3,\lambda_1+\lambda_2+\lambda_3+\lambda_4].
$$
Let
 $T:[0,\vert\boldsymbol{\lambda}\vert]\to [0,\vert\boldsymbol{\lambda}\vert]$ be the map (called \textrm{isometric model}) defined by
\begin{equation}\label{formulaT}
T(x)=\begin{cases} -x +\lambda_1 + \lambda_3 & \textrm{if} \quad  x\in I_1\\
-x +\lambda_1 + |\boldsymbol{\lambda}| & \textrm{if} \quad x\in I_2 \\
-x+\lambda_1+\lambda_2 + \lambda_3 & \textrm{if}\quad x\in I_3 \\
-x + \lambda_1 + \lambda_3 +  |\boldsymbol{\lambda}|  & \textrm{if} \quad x\in I_4
\end{cases}.
\end{equation}
According to the definition given in \cite{AK1980}, we have that $T$ is a {\it $4$-interval exchange transformation with flips} ($4$-IET with flips). In fact, it can be easily verified that $T$ is one-to-one on $(0,\vert\boldsymbol{\lambda}\vert]$, $T\vert_{I_i}$ is an isometry $(i=1,2,3,4)$ and $T$ reverts the orientation of one (in fact, all) of the intervals $I_1,I_2,I_3,I_4$. We denote by $O_T(x)=\{x,T(x),T^2(x),\ldots\}$ the \textit{$T$-orbit of $x\in [0,\vert\boldsymbol{\lambda}\vert]$}. We say that $T$ is \textit{topologically transitive} if it has a dense orbit; \textit{minimal} if every $T$-orbit is dense; \textit{uniquely ergodic} if the (normalized) Lebesgue measure on $[0,\vert\boldsymbol{\lambda}\vert]$ is the only $T$-invariant Borel probability measure. 

Our first result is the following.

\begin{theorem}\label{thmT} The map $T$  defined in $(\ref{formulaT})$ is minimal and uniquely ergodic.
\end{theorem}
The example given in Theorem \ref{thmT} is rare. Typically, an $n$-IET with flips has an interval formed by periodic orbits and, therefore, is not minimal (see \cite{AN1989}). This situation is completely different in the case of IETs without flips, also called \textit{standard} IETs. The simplest example is the  rotation of the circle $R_{\alpha}:[0,1)\to [0,1)$ defined by  $R_{\alpha}(x)=x+\alpha \,(\textrm{mod}\,\, 1)$, where $0<\alpha<1$. It
 can be written as the standard $2$-IET $T_{\alpha}:[0,1]\to [0,1]$ defined by $T_{\alpha}(x)= x+1-\alpha$ if $x\in [0,\alpha)$ and  $T_{\alpha}(x)= x-\alpha$ if $x\in [\alpha,1]$. It is widely known that when $\alpha$ is irrational, $R_{\alpha}$ and $T_\alpha$ are minimal and uniquely ergodic. Concerning standard irreducible $n$-IETs with $n\ge 2$, Keane's conjecture,
  answered in the affirmative by many authors (see \cite{MB1985,SK1985,HM1982,MR1981,WV1982}), states that
  such maps are typically minimal and uniquely ergodic. 
  
 To state our main result, we need some more definitions. Let
$$p_1=0, \quad p_2=T(\lambda_1+\lambda_2), \quad  p_3=T(\lambda_1+\lambda_2+\lambda_3), \quad   p_4=\vert\boldsymbol{\lambda}\vert.
$$ 
For $i,j\in\{1,2,3,4\}$, let
\begin{equation}\label{kij}
{K}_{ij}=\{k\ge 0: T^k(p_j)\in I_i\}, \quad c_{ij}=\sum_{k\in{K}_{ij}} \frac{1}{2^k}.
\end{equation}
Let
$$M=
\begin{pmatrix} 
c_{11}-c_{14} & c_{12}-c_{14} & c_{13}-c_{14} \\
c_{21}-c_{24} & c_{22} -c_{24}& c_{23} - c_{24}\\
c_{31} - c_{34}& c_{32} - c_{34} & c_{33} - c_{34}  \\
\end{pmatrix}.
$$
Let $u_4=1-u_1-u_2-u_3>0$, where $u_1,u_2,u_3>0$ is the unique solution of the linear system
$$\begin{pmatrix} u_1 \\ u_2 \\ u_3
\end{pmatrix}=
\frac12 M \begin{pmatrix} 
-1 & -1 & -1\\
0 & 1 & 0\\
0 & 0 & 1 \\
\end{pmatrix}
\begin{pmatrix}
 u_1 \\ u_2 \\ u_3
\end{pmatrix}+\frac12M
\begin{pmatrix}
1 \\ 0 \\ 0
\end{pmatrix}+\frac12
\begin{pmatrix}
 c_{14} \\ c_{24}\\  c_{34}
\end{pmatrix}.
$$
Let $$z_1=u_1,\quad z_2=u_1+u_2,\quad z_3=u_1+u_2+u_3.$$

In what follows, we say that a map $f:[0,1]\to[0,1]$ is \textit{topologically semiconjugate} to the isometric model $T:[0,\vert\boldsymbol{\lambda}\vert]\to [0,\vert\boldsymbol{\lambda}\vert]$ if there exists a continuous, surjective, nondecreasing map $h:[0,1]\to [0,\vert\boldsymbol{\lambda}\vert]$ such that $h\circ f=T\circ h$.

Now we state our main result.

\begin{theorem}\label{omr} Let $d_1,d_2,d_3>0$ be defined by
$$d_1=\frac{1}{3z_1}-1\cong 0.213841, \quad d_2=\dfrac{2-3z_2}{3z_2-1}\cong 4.036935,\quad d_3=\dfrac{3-3z_3}{3z_3-2}\cong 1.428826.$$
Then for any switched server system with parameters  $d_{ij}$ satisfying $(\ref{dij})$ the following statements are true:
\begin{itemize}
\item [$(a)$] The switched server system has no periodic orbit; 
\item [$(b)$] The Poincar\'e map $F:\partial\Delta\to\partial\Delta$ of the system is topologically semiconjugate to $T$;
\item [$(c)$] $\omega_F(\mathbf{v})$ is a Cantor set for every $\mathbf{v}\in \partial\Delta$;
\item [$(d)$] The frequency $\textrm{freq}\,(i)$ with which the server is connected to the tank $i$										  at the switching times is
$$ \textrm{freq}\,(1)=\dfrac{\lambda_3}{\vert\boldsymbol{\lambda}\vert}\cong 34.44\%,\quad \textrm{freq}\,(2)=\frac{\lambda_1+\lambda_4}{\vert\boldsymbol{\lambda}\vert}\cong 41.82\%,\quad 
\textrm{freq}\,(3)=\frac{\lambda_2}{\vert\boldsymbol{\lambda}\vert}\cong 23,72\%.
$$
\end{itemize}
\end{theorem}

In Theorem \ref{omr}, the item (d) follows from the item (b), from Theorem \ref{thmT} and from the version of Birkhoff's Ergodic Theorem for uniquely ergodic transformations (see \cite[Proposition 4.1.13]{AK1980}). The itens (a) and (d) are also confirmed by numerical simulations using the $R$ programming language. It is also worth mentioning that the matrices $P$ and $Q$ were obtained by using Rauzy induction.

\section{Poincar\'e maps of switched server systems and the proof of Lemma \ref{lmint}}

We keep all the notations given in the previous sections.

\begin{proof}[Proof of Lemma \ref{lmint}] Let $d_1,d_2,d_3>0$ be given. Let the switched server system parameters $d_{ij}$ be chosen according to $(\ref{dij})$.  
Let $0\le t_1<t_2\ldots$ denote the switching times. If at the switching time $t_m$ the server is connected to the tank $j$, then it keeps connected to the tank $j$ during the time-interval $[t_m,t_{m+1})$. Moreover,
\begin{equation}\label{tm1-38}
t_{m+1}-t_m=\dfrac{v_{j}(t_m)}{\rho-\rho_j}=\dfrac{v_{j}(t_m)}{1-\frac{1}{3}}=\dfrac{3}{2} v_j(t_m).
\end{equation}
For every $m\ge 1$ and $t_m\le  t \le  t_{m+1}$, the level $v_k(t)$ of any tank $k\in \{1,2,3\}$ is determined by the set of linear equations
\begin{equation}\label{tm2-38}
v_k(t)=\begin{cases}
v_k(t_m)+\dfrac{1}{3}(t-t_m)\,\quad\textrm{if}\,\,\,\, k\ne j\\[0.12in]
v_j(t_m)-\dfrac23(t-t_m)\quad
\textrm{if}\,\,\,\,  k=j
\end{cases},
 \end{equation}
 where $j$ is the position of the server at the time $t_m$. 
 
 The equation (\ref{tm2-38}) shows that the state ${\bf v}(t)=\big(v_1(t),v_2(t),v_3(t)\big)$ of the system at any time $t\in [t_m,t_{m+1})$ describes the position of a particle that moves with constant velocity. More precisely, when the particle hits $\partial\Delta$ at the switching time $t_m$, it takes the velocity $\mathbf{v}'(t_m+)$ and moves with such velocity till it hits the boundary again, at the time $t_{m+1}$, when then the velocity changes to $\mathbf{v}'(t_{m+1}+)$. In this way, $t\in [0,\infty)\mapsto \mathbf{v}(t)\in\Delta$ is the trajectory of a
  \textit{pseudo billiard}.  By sampling the system at the consecutive switching times $t_1$ and $t_2$, we obtain the Poincar\'e map $F:\partial\Delta\to\partial\Delta$ induced by the flow on the boundary $\partial\Delta$ of $\Delta$. 
More specifically, considering $m=1$ in (\ref{tm1-38}) and (\ref{tm2-38}), $t=t_2$ in (\ref{tm2-38}), and $(v_1,v_2,v_3)= (v_1(t_1),v_2(t_1),v_3(t_1))\in\partial\Delta$ yield
 \begin{equation}\label{tm3-38}
\big(F(v_1,v_2,v_3)\big)_k=v_k(t_{2})=\begin{cases}
v_k+\dfrac{1}{2}v_j& \textrm{if}\,\, k\ne j\\
0 &
\textrm{if}\,\,  k=j
\end{cases},
\end{equation}
where $j$ is the position of the server at the time $t_1$. Notice that if $i\neq j$ denotes the empty tank number at the time $t_1$, then $d_{ij}v_j=\max\,\{d_{ik} v_k: 1\le k\le 3\}$, that is, at the time $t_1$, the server begins emptying the tank $j$ with the largest scaled volume $d_{ij} v_j$. Now we will find a piecewise-defined formula for $F$. Let $$\mathbf{e}_1=(1,0,0),\quad \mathbf{e}_2=(0,1,0),\quad \mathbf{e}_3=(0,0,1).$$ 
Given  $\mathbf{p},\mathbf{q}\in\mathbb{R}^3$, let the line segments $[\mathbf{p},\mathbf{q}]$, $(\mathbf{p},\mathbf{q})$, $[\mathbf{p},\mathbf{q})$ and $(\mathbf{p},\mathbf{q}]$ be defined as usual, for instance,
$$[\mathbf{p},\mathbf{q}]=\{(1-\alpha)\mathbf{p}+\alpha\mathbf{q}:0\le\alpha\le 1\},\quad (\mathbf{p},\mathbf{q})=\{(1-\alpha)\mathbf{p}+\alpha\mathbf{q}:0 <\alpha < 1\}.$$  Notice that
$$ \partial \Delta=[\mathbf{e}_2,\mathbf{e_3}]\cup [\mathbf{e}_3,\mathbf{e_1}] \cup [\mathbf{e}_1,\mathbf{e_2}].
$$
Moreover,
\begin{equation}\label{iff1-38}
\begin{cases}(v_1,v_2,v_3)\in  [\mathbf{e}_2,\mathbf{e_3}] \iff v_1=0 \\
(v_1,v_2,v_3)\in  [\mathbf{e}_3,\mathbf{e_1}] \iff v_2=0 \\
(v_1,v_2,v_3)\in  [\mathbf{e}_1,\mathbf{e_2}] \iff v_3=0
\end{cases}.
\end{equation}
Now let us consider the decomposition of $\partial\Delta$ given by (see  Figure \ref{border-38}):
$$ \partial\Delta= [\mathbf{r}_1,\mathbf{e}_3]\cup [\mathbf{e}_3,\mathbf{r}_2) \cup [\mathbf{r}_2,\mathbf{e}_1] \cup [\mathbf{e}_1,\mathbf{r}_3) \cup [\mathbf{r}_3,\mathbf{e}_2] \cup [\mathbf{e}_2,\mathbf{r}_1),
$$
where
$$\mathbf{r}_1=\frac{d_{13}}{d_{12}+d_{13}}\mathbf{e}_{2}+\frac{d_{12}}{d_{12}+d_{13}}\mathbf{e}_3,\,\,\mathbf{r}_2=\frac{d_{21}}{d_{23}+d_{21}}\mathbf{e}_3+\frac{d_{23}}{d_{23}+d_{21}}\mathbf{e}_1,\,\,\mathbf{r}_3=\frac{d_{32}}{d_{31}+d_{32}}\mathbf{e}_1+\frac{d_{31}}{d_{31}+d_{32}}\mathbf{e}_2.$$
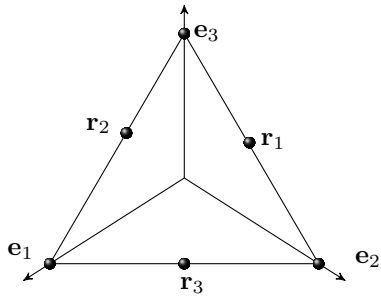
\begin{figure}[ht!]
\begin{tikzpicture}[tdplot_main_coords, scale=2.5]
   \draw[-stealth'] (0,0,0) -- (1.2,0,0);
   \draw[-stealth'] (0,0,0) -- (0,1.2,0) ;   
   \draw[-stealth'] (0,0,0) -- (0,0,1.2);
   \draw [thin, fill opacity=0.5]
          (1,0,0) -- (0,1,0) -- (0,0,1) -- cycle;
   \draw (0.95,-0.1) node[left] {$\mathbf{e}_1$};
   \draw (0,1.2,0.15) node[right] {$\mathbf{e}_2$};
   \draw (0,0.,1.) node[right] {$\mathbf{e}_3$};
   \draw (0,0.5,0.55) node[right] {$\mathbf{r}_1$};
   \draw (-0.2,-1,-0.35) node[right] {$\mathbf{r}_2$};
    \draw (0.55,0.45,-0.15) node[right] {$\mathbf{r}_3$};
  \shade[ball color = black] (0,0,1) circle (0.03cm);  
    \shade[ball color = black] (1,0,0) circle (0.03cm);  
     \shade[ball color = black] (0,1,0) circle (0.03cm);  
      \shade[ball color = black] (0.015,0.5,0.55) circle (0.03cm);  
 \shade[ball color = black] (-0.3,-0.73,-0.3) circle (0.03cm); 
  \shade[ball color = black] (0.5,0.5,0) circle (0.03cm);  
   \end{tikzpicture}
   \caption{Partition of  $\partial\Delta$}\label{border-38}
\end{figure}

Let $(v_1,v_2,v_3)\in (\mathbf{r}_1,\mathbf{e}_3]$, then $v_1=0$, that is, $i=1$. Moreover, 
$$ v_3 > \dfrac{d_{12}}{d_{12}+d_{13}},\, v_2< \dfrac{d_{13}}{d_{12}+d_{13}} \quad\textrm{and}\quad d_{13} v_3 >  \frac{d_{13} d_{12}}{d_{12}+d_{13}}=\dfrac{d_{12}d_{13}}{d_{12}+d_{13}}> d_{12} v_2,
$$
implying that the tank $3$ has the largest scaled volume, that is, $j=3$. Proceeding likewise with respect to $[\mathbf{e}_3,\mathbf{r}_2)$,  $ [\mathbf{r}_2,\mathbf{e}_1]$, etc., and using the convention that $l$ is right-continuous (see Introduction),
we reach the following conclusion.
\begin{equation}\label{iff2-38}
\begin{cases} (v_1,v_2,v_3)\in   [\mathbf{r}_1,\mathbf{e}_3]\cup [\mathbf{e}_3,\mathbf{r}_2)  \iff j=3 \\
(v_1,v_2,v_3)\in  [\mathbf{r}_2,\mathbf{e}_1] \cup [\mathbf{e}_1,\mathbf{r}_3)  \iff j=1 \\
(v_1,v_2,v_3)\in   [\mathbf{r}_3,\mathbf{e}_2] \cup [\mathbf{e}_2,\mathbf{r}_1) \iff  j=2
\end{cases}.
\end{equation}
Putting together (\ref{tm3-38}), (\ref{iff1-38}) and (\ref{iff2-38}), we reach
\begin{equation}\label{Fgrande-38}
 F(v_1,v_2,v_3)=\begin{cases} \left(\frac12 v_2,0,v_3+\frac12 v_2\right) & \textrm{if}\,\,\, (v_1,v_2,v_3)\in [\mathbf{e}_2,\mathbf{r}_1) \\
 \left(v_1+\frac12 v_3,v_2+\frac12 v_3,0\right) & \textrm{if}\,\,\, (v_1,v_2,v_3)\in [\mathbf{r}_1,\mathbf{e}_3]\cup  [\mathbf{e}_3,\mathbf{r}_2) 
   \\
  \left( 0, v_2+\frac12 v_1, v_3+\frac12 v_1\right) & \textrm{if}\,\,\, (v_1,v_2,v_3)\in [\mathbf{r}_2,\mathbf{e}_1]\cup [\mathbf{e}_1,\mathbf{r}_3) \\
    \left( v_1+\frac12 v_2, 0, \frac12 v_2\right) & \textrm{if}\,\,\, (v_1,v_2,v_3)\in [\mathbf{r}_3,\mathbf{e}_2]
\end{cases}.
\end{equation}
Let $\varphi:[0,1]\to\partial\Delta$ be the anticlockwise arc-length parametrization of $\partial\Delta$ (see Figure \ref{fig1-38}). More precisely,  
let
 \begin{equation}
\varphi(t)=\begin{cases}\label{defvarphi-38}
(1-3t)\mathbf{e}_2+3t\mathbf{e}_3&\,\,\,\textrm{if}\,\,\, t\in \left[0,\frac{1}{3}\right)\\
(2-3t)\mathbf{e}_3+(3t-1)\mathbf{e}_1 &\,\,\,\textrm{if}\,\,\, t\in \left[\frac13,\frac23\right)\\
(3-3t)\mathbf{e}_1+(3t-2)\mathbf{e}_2 &\,\,\,\textrm{if}\,\,\, t\in \left[\frac23,1\right]\\
\end{cases}.
\end{equation}
\begin{center}
 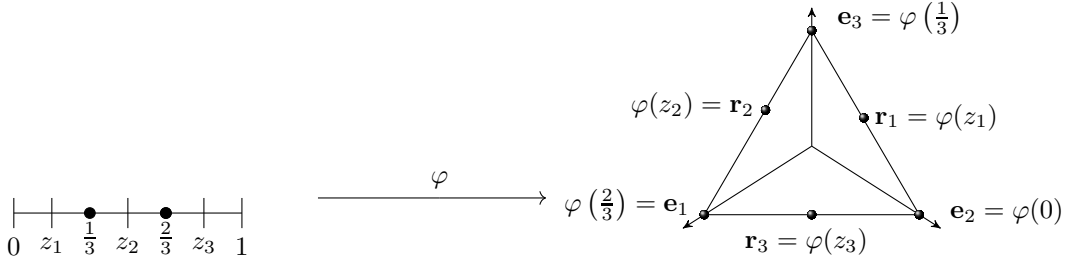
\begin{figure}[h!]
  \begin{tabular}{cc}
  %
  %

\begin{tikzpicture}[scale=2]
\draw[fill=black] (0.5,0.2) circle (1pt);
\draw[fill=black] (1,0.2) circle (1pt);
\draw[fill=black] (1,0.2) circle (1pt);
\draw (0,0.2) -- (1.5,0.2); 
\draw   (0,0.3)--(0,0.1) node[below]{$0$};
\draw   (0.5,0.2) node[below]{$\frac13$};
\draw   (0.25,0.3)--(0.25,0.1) node[below]{$z_1$};
\draw   (0.75,0.3)--(0.75,0.1) node[below]{$z_2$};
\draw   (1.25,0.3)--(1.25,0.1) node[below]{$z_3$};
\draw   (1,0.2) node[below]{$\frac23$};
\draw   (1.5,0.3)--(1.5,0.1) node[below]{$1$};
\draw   (2,0.3)--(2.8,0.3) node[above]{$\varphi$};
\draw[->]   (2.8,0.3)--(3.5,0.3);
\draw plot[id=x] function{x*x};
\end{tikzpicture}

  \begin{tikzpicture}[tdplot_main_coords, scale=2]
   \draw[-stealth'] (0,0,0) -- (1.2,0,0);
   \draw[-stealth'] (0,0,0) -- (0,1.2,0) ;   
   \draw[-stealth'] (0,0,0) -- (0,0,1.2);
   \draw [thin, fill opacity=0.5]
          (1,0,0) -- (0,1,0) -- (0,0,1) -- cycle;
   \draw (0.95,-0.1) node[left] {$\varphi\left(\frac23\right)=\mathbf{e}_1$};
   \draw (0,1.2,0.15) node[right] {$\mathbf{e}_2=\varphi(0)$};
   \draw (0,0.15,1.2) node[right] {$\mathbf{e}_3=\varphi\left(\frac13\right)$};
    \draw (0,0.5,0.55) node[right] {$\mathbf{r}_1=\varphi(z_1)$};
   \draw (-0.2,-0.65,-0.15) node[left] {$\varphi(z_2)=\mathbf{r}_2$};
    \draw (0.55,0.5,-0.2) node {$\mathbf{r}_3=\varphi(z_3)$};
  
    \shade[ball color = black] (0,0,1) circle (0.03cm);  
    \shade[ball color = black] (1,0,0) circle (0.03cm);  
     \shade[ball color = black] (0,1,0) circle (0.03cm);  
      \shade[ball color = black] (0.015,0.5,0.55) circle (0.03cm);  
 \shade[ball color = black] (-0.3,-0.73,-0.3) circle (0.03cm); 
  \shade[ball color = black] (0.5,0.5,0) circle (0.03cm);  
  \end{tikzpicture}

   \end{tabular}
  \caption{The arc-length parametrization of $\partial\Delta$} 
  \label{fig1-38}
 \end{figure}
\end{center}
The inverse of $\varphi$ is defined by
\begin{equation}\label{thei-38}
\varphi^{-1}(\mathbf{p})=\begin{cases} \dfrac{1}{3\sqrt{2}}\Vert \mathbf{p}-\mathbf{e}_2\Vert\quad\quad\,\,\,\,\,\textrm{if}\quad \mathbf{p}\in [\mathbf{e}_2,\mathbf{e}_3] \\[0.15in]
\dfrac{1}{3\sqrt{2}} \Vert \mathbf{p}-\mathbf{e}_3\Vert +\dfrac{1}{3}\quad \textrm{if}\quad \mathbf{p}\in [\mathbf{e}_3,\mathbf{e}_1] \\[0.1in] 
\dfrac{1}{3\sqrt{2}} \Vert \mathbf{p}-\mathbf{e}_1\Vert +\dfrac{2}{3}\quad  \textrm{if}\quad \mathbf{p}\in [\mathbf{e}_1,\mathbf{e}_2] 
\end{cases}.
\end{equation}
It follows from (\ref{Fgrande-38}), (\ref{defvarphi-38}), and (\ref{thei-38}) that the map $f=\varphi^{-1}\circ F\circ \varphi$ is given by
$$
f(z)=
\begin{cases} -\dfrac12 z + \dfrac12 & \textrm{if} \,\, z\in [z_0,z_1)\\[0.5em]
-\dfrac12 z + 1 & \textrm{if} \,\, z\in [z_1,z_2)\\[0.5em]
-\dfrac12 z + \dfrac12 & \textrm{if} \,\, z\in [z_2,z_3) \\[0.5em]
-\dfrac12 z + 1 & \textrm{if} \,\, z\in [z_3,z_4] \\[0.5em]
\end{cases},
$$
where
\begin{equation}\label{y1234}
z_0=0,\quad z_1=\dfrac{d_{12}}{3(d_{12}+d_{13})},\quad z_2=\dfrac{d_{23}}{3(d_{23}+d_{21})}+\dfrac{1}{3}, \quad
z_3=\dfrac{d_{31}}{3(d_{31}+d_{32})}+\dfrac{2}{3},\quad z_4=1.
\end{equation}
By (\ref{dij}),  we have that $(\ref{y1234})$ is equivalent to
$(\ref{y123})$, hence $f(z)=f_{d_1,d_2,d_3}(z)$ for every $z\in [0,1]$. This concludes the proof of 
Lemma \ref{lmint}.
\end{proof}
 
\section{Interval exchange transformations (IETs)}\label{tim}

In this section, we gather some results related to the construction of topologically transitive IETs. We will use them in the next section in the proof of Theorem \ref{thmT}.

Let $a>0$ and $I=[0,a]$. Following \cite{AK1980},
we say that $T:I\to I$ is an \textit{$n$-interval exchange transformation} ($n$-IET) if there exist a partition of $I$ into intervals
$I_1,I_2,\ldots,I_n$ with endpoints $\{x_0,x_1\}$, $\{x_1,x_2\}$, \ldots, $\{x_{n-1},x_n\}$ satisfying $0=x_0<x_1<\cdots<x_n=a$
 such that 
$T$ is one-to-one on $I{\setminus}\{0\}$ and $T\vert_{I_i}$ is an isometry ($i=1,2,\ldots,n$). 
The vector $\lambda=(\lambda_1,\lambda_2,\ldots,\lambda_n)$ with $\lambda_i=x_{i}-x_{i-1}$ is called the \textit{length vector}. 
Notice that there exist $\varepsilon_i\in \{-1,1\}$ and $b_i\in \mathbb{R}$ $(i=1,2,\ldots,n)$ such that
\begin{equation}\label{Ti}
T(x)=T_i(x):=\varepsilon_i x + b_i \quad\textrm{for all}\quad x\in (x_{i-1},x_i) \quad (i=1,2,\ldots,n).
\end{equation}
If $\varepsilon_i=1$ ($i=1,2,\ldots,n$), then we say that $T$ is \textit{standard}, otherwise we say that $T$ has \textit{flips}.
We assume that $\mathcal{D}(T)=\{x_1,x_2,\ldots,x_{n-1}\}$ is the set of discontinuities of $T$, otherwise $T$ would be an $m$-IET with $m<n$.
  
  \subsection{Poincar\'e maps of IETs}\label{Frm}\rule[0pt]{0pt}{5pt}\\ 
  
  Let $0=x_0<x_1<\ldots<x_n=a$ and let 
$T:I\to I$ be an $n$-IET  defined on $I=[0,a]$ with set of discontinuities $\mathcal{D}(T)=\{
x_1,x_2,\ldots,x_{n-1}\}$. 
 
 \begin{definition}[T-tower] Given $r\ge 1$, we say that $\{J,T(J),\ldots,T^{r-1}(J)\}$ is a \textit{$T$-tower} if $J,T(J),\ldots,T^{r-1}(J)$ are pairwise disjoint open intervals. Each interval $T^k(J)$, $0\le k\le r-1$, is called a \textit{floor}.
 \end{definition}
It is an elementary fact that all the floors in a $T$-tower have the same length $\vert J\vert$. In this way, $r\le \vert I\vert/\vert J\vert$. Equivalently, a family $\{J_1,J_2,\ldots,J_{r}\}$ of pairwise disjoint open intervals is a $T$-tower if there exists a permutation $\tau:\{1,\ldots,r\}\to \{1,\ldots,r\}$ such that
$J_{\tau(i+1)}=T(J_{\tau(i)})$ for every $1\le i\le r-1$.

The following result is a consequence of the injectivity of $T$ on $(0,a)\subset I{\setminus}\{0\}$.

\begin{lemma}\label{lem31} If $\{J,T(J),\ldots, T^{r-1}(J)\}$ is a $T$-tower with $T^{r-1}(J)\cap\mathcal{D}(T)=\emptyset$, then either
$T^r(J)\cap J \neq\emptyset$ or $\{J,T(J),\ldots, T^{r}(J)\}$ is a $T$-tower.
\end{lemma}
\begin{proof} Set $U=(0,a){\setminus}\mathcal{D}(T)$. Since $T^{r-1}(J)$ is an open subinterval of $U$ and $T$ is an isometry on each connected component of $U$, we have that $T^{r}(J)=T\big(T^{r-1}(J)\big)$ is an open interval. Without loss of generality, we assume that $r\ge 2$.
Clearly,  since $T$ is injective on $I{\setminus}\{0\}$ and $J,T(J),\ldots, T^{r-1}(J)$ are pairwise disjoint open intervals,
we have that  $T^r(J)\cap T^k(J)=\emptyset$ for all
$1\le k\le r-1$, which concludes the proof. \end{proof}
 Let  $0<a'<a$ and $I'=[0,a']$. Given $x\in I$, let
$N(x)\in \mathbb{N}\cup \{\infty\}$ be defined by 
\begin{equation}\label{defN}
N(x)=\inf\, \{N\ge 1: T^N(x)\in I' \},
\end{equation}
 where $\inf\emptyset=\infty$. The map $T':\textrm{dom}\,(T')\to I'$, where $\textrm{dom}\,(T')=\{x\in I': N(x)<\infty\}$ and $$T'(x)=T^{N(x)}(x)=\underbrace{T\circ T\circ\ldots\circ T}_{N(x)\,\textrm{times}}(x)$$ is called the \textit{Poincar\'e map of $T$ on $I'$}. 
 \begin{definition}[Admissible interval]\label{admint}  The interval $I'$ is \textit{admissible} if there exist  
$0=x_0'<x_1'<\ldots<x_n'=a'$ such that $N(x_i')<\infty$ for every $1\le i\le n$ and the set
 $B=\bigcup_{i=1}^{n} \left\{x_i',T(x'_i),\ldots,T^{N(x'_i)-1}(x'_i)\right\}$ satisfies
 \begin{itemize}
  \item [$(H1)$] $B\supset \mathcal{D}(T)$;
 \item [$(H2)$] $a'\in T(B)$.
   \end{itemize}
   \end{definition}
   Henceforth, we will assume that $I'$ is an admissible interval.
   \begin{lemma}\label{lem32}
    Let $K\subset I{\setminus B}$ be an open interval. Then $K\cap\mathcal{D}(T)=\emptyset$. Moreover,
  one of the following alternatives happens: 
 \begin{itemize} 
 \item [$(i)$] $T(K)$ is an open subinterval of $I'$;
 \item [$(ii)$] $T(K)\cap I'=\emptyset$ and $T(K)$ is an open subinterval of $I{\setminus B}$.
 \end{itemize}
 \end{lemma}
 \begin{proof}  By $(H1)$, $\mathcal{D}(T)\subset B$, thus $K\cap \mathcal{D}(T)=\emptyset$ and $T(K)$ is an open interval. Let $1\le i\le n$. Then $x_i'\neq 0$. By (\ref{defN}), $T^k(x_i')\not\in I'$ for all $1\le k\le N(x_i')-1$. In this way, $K\cup B\subset I{\setminus}\{0\}$. Hence, $T$ is injective on $K\cup B$, then $T(B)\cap T(K)=\emptyset$. Now, by $(H2)$,
 $a'\not\in T(K)$, thus  either $T(K)\subset I'$ or $T(K)\cap I'=\emptyset$. In the latter case,
 $T(K)\cap B\subset B{\setminus} I'\subset T(B)$, which yields $T(K)\subset I{\setminus} B$.     \end{proof}

  \begin{lemma}\label{lem33} Let $J$ be an open subinterval of $I'{\setminus} \{x'_1,\ldots,x'_{n-1}\}$, then there exists $r\ge 1$ such
  that \linebreak $\{J, T(J),\ldots,T^{r-1}(J)\}$ is a $T$-tower, $\bigcup_{k=0}^{r-1} T^k(J)\subset I{\setminus} \{x_0,\ldots,x_n\}$, $I'\cap\bigcup_{k=1}^{r-1} T^k(J)=\emptyset$
  and 
  $T^r(J)$ is a subinterval of $I'$. 
  \end{lemma}
  \begin{proof} By the definition of $B$, we have that $B\cap I'=\{x_1',\ldots,x_{n}'\}$, thus $J\subset I{\setminus }B$ and $T(J)$ is an open interval by (H1). If $T(J)\subset I'$, then we take $r=1$ and we are done. 
  Otherwise, applying Lemma \ref{lem32} with
  $K=J$ yields $I'\cap T(J)=\emptyset$ and $T(J)\subset I{\setminus B}$. Moreover, in this case, we have that the set  
  $$A=\left\{\alpha\ge 1: \{J,T(J),\ldots,T^{\alpha-1}(J)\}\,\,\textrm{is an}\,\, \alpha\textrm{-tower with}\,\, I'\cap\bigcup_{k=1}^{\alpha-1} T^{k}(J)=\emptyset\right\}.$$
  is a non-empty subset of $\left[1,\frac{\vert I\vert }{\vert J\vert}\right]$. By applying Lemma \ref{lem32} finitely many times, we can prove that  $r=\max A$ works.
 \end{proof}

  \begin{proposition}\label{277} Let $T:I\to I$ be an $n$-IET and $I'\subset I$ be an admissible interval for $T$. Then,  for each $1\le i\le n$, there exist $r_i\ge 1$ and a  word $i_0 i_1 \ldots i_{{r_i}-1}$ over the alphabet $\mathcal{A}=\{1,\ldots,n\}$
    such that the interval $J_i=(x_{i-1}',x_i')$ satisfies
  \begin{itemize}
\item [$(A1)$] $\{ J_i$, $T(J_i),\ldots,T^{r_{i}-1}(J_i)\}$ is a $T$-tower with $I'\cap\bigcup_{k=1}^{r_i-1} T^k(J_i)=\emptyset$;
\item [$(A2)$] $T^{r_i}(J_i)$ is an open subinterval of $I'$;
\item [$(A3)$] $T^k(J_i)\subset (x_{i_k-1},x_{i_k})$ for every $0\le k\le r_i-1$;
\item [$(A4)$] $N(x)=r_i$ for all $x\in J_i$.
 \end{itemize}
 Moreover, the intervals $T^k(J_i), 0\le k\le r_i-1,\,\, 1\le i\le n$, are pairwise disjoint.
  \end{proposition}
  \begin{proof}Applying Lemma \ref{lem33} with $J=J_i$ yields $(A1)$, $(A2)$ and $(A3)$. The item $(A4)$ follows from $(A1)$ and $(A2)$. We claim that $T^k(J_i),$ $0\le k\le r_i-1,\,\, 1\le i\le n$ are pairwise disjoint. Otherwise, by $(A1)$, there exist $i\neq j$, $0\le k_i\le r_i-1$, $0\le k_j\le r_j-1$ with $k_i\le k_j$ such that
  $T^{k_i}(J_i)\cap T^{k_j}(J_j)\neq\emptyset$. By the injectivity of $T$ on $(0,a)$, we obtain that $J_i\cap T^{k_j-k_i}(J_j)\neq\emptyset$, which is a contradiction since $J_i\subset I'$ while $T^{k_j-k_i}(J_j)\cap I'=\emptyset$. 
   \end{proof}
 
 In Proposition \ref{277}, the word $i_0 i_1 \ldots i_{r_i-1}$ is the \textit{symbolic itinerary} of the $T$-tower $\{ J_i$, $T(J_i),\ldots,T^{r_{i}-1}(J_i)\}$. Concerning the next three corollaries, we let $J_i$, $r_i$ and $i_0 i_1\ldots i_{{r_i}-1}$,  be as in the statement of Proposition \ref{277}.

\begin{corollary}\label{xy17} Let $T:I\to I$ be an $n$-IET and $I'\subset I$ be an admissible interval for $T$. Then the Poincar\'e map $T'$ of $T$ on $I'$ is the $n'$-IET, $n'\le n$, defined by
$$T'(x)=T_{i_{r_i-1}}\circ \cdots\circ T_{i_1}\circ T_{i_0}(x)\quad \textrm{if}\quad x\in (x_{i-1}',x_i'),$$
where $T_i:\mathbb{R}\to\mathbb{R}$ is the affine map defined by $(\ref{Ti})$. Notice that $\mathcal{D}(T')\subset\{x_1',\ldots,x_{n-1}'\}$.    \end{corollary} 
 \begin{definition}[Exhaustive family] The family of $T$-towers $\left\{ J_i, T(J_i),\ldots, T^{r_i-1}(J_i)\right\}$, $1\le i\le n$, is \textit{exhaustive} if
  all the floors are pairwise disjoint and
  $I{\big\backslash}\bigcup_{i=1}^n \bigcup_{k=0}^{r_i-1}T^k(J_i)$ is a finite set.
 \end{definition}

  \begin{corollary}\label{xyf} Let $T:I\to I$ be an $n$-IET and $I'\subset I$ be an admissible interval for $T$. Suppose that
 \begin{equation*}
 \leqno{\phantom{aa}(H3)} \,\,\sum_{i=1}^n r_i \vert J_i\vert=\vert I\vert,
 \end{equation*}
 then the family of $T$-towers $\{ J_i$, $T(J_i),\ldots,T^{r_{i}-1}(J_i)\}$, $1\le i\le n$, in Proposition \ref{277}, is exhaustive.
  \end{corollary}
  \begin{proof} In fact, in this case, by Proposition \ref{277}, $S=I{\big\backslash}\bigcup_{i=1}^n \bigcup_{k=0}^{r_i-1}T^k(J_i)$ is the union of finitely many compact intervals and has Lebesgue measure zero, which implies that $S$ is a finite set.
\end{proof}

\begin{corollary}\label{48} Let $T:I\to I$ be an $n$-IET and $I'\subset I$ be an admissible interval for $T$ such that $(H3)$ holds.
If $T'$ is topologically transitive, so is $T$.
\end{corollary}
\begin{proof} By Corollary \ref{xyf},
$I{\setminus}\bigcup_{i=1}^n\bigcup_{k=0}^{r_i-1} T^k(J_i)$ is a finite set. Moreover, by (A3) of Proposition \ref{277},
$x\in J_i\mapsto T^k(x)\in T^k(J_i)$ is an isometry for every $0\le k\le r_i-1$ and $1\le i\le n$. Since $I'$ is the closure of $\cup_{i=1}^n J_i$, any $T'$-orbit dense in $I'$ corresponds to a $T$-orbit dense in $I$.
\end{proof}
\subsection{Self-similar IETs}\rule[-5pt]{0pt}{5pt}\\

Let $I'\subset I$ be an admissible interval for $T$. By Corollary \ref{xy17}, the Poincar\'e map $T':I'\to I'$ is an $n'$-IET with set of discontinuities $\mathcal{D}(T')\subset \{x_1',\ldots,x_{n-1}'\}$. 
\begin{definition}[self-similar IET] Let $T:I\to I$ be an $n$-IET and $I'\subset I$ be an admissible interval for $T$. We say that \textit{$T$ is self-similar on $I'$} if 
$T'=L\circ T\circ L^{-1}$ on $I'{\setminus\{x_0',\ldots,x_n'\}}$, where $L:I\to I'$ is the affine bijection $x\mapsto \frac{a'}{a}x$. 
\end{definition}
In other words, $T$ is self-similar on $I'$ if
 $\mathcal{D}(T')= \{x_1',\ldots,x_{n-1}'\}$ and 
$T'$ is a rescaled copy of $T$. In particular, we have that $\mathcal{D}(T')=L\big(\mathcal{D}(T)\big)$.  

Denote by $\mathcal{A}^*$ the set of (finite) words over the alphabet $\mathcal{A}=\{1,2,\ldots,n\}$. By $(A3)$ in Proposition \ref{277}, to the pair $(T,I')$, we can associate the map $\sigma:\mathcal{A}\to \mathcal{A}^*$
 defined by $\sigma(i)=i_0i_1,\ldots i_{r_i-1}$ called the \textit{substitution associated with} $(T,I' )$. In this way, the substitution $\sigma$ assigns to each letter $i\in\mathcal{A}$, the symbolic itinerary of the $T$-tower $\{ J_i$, $T(J_i),\ldots,T^{r_{i}-1}(J_i)\}$. By means of the concatenation operation, we can consider ${\sigma}$ as a self-map of
$\mathcal{A}^*$. The \textit{matrix associated with $(T,I')$} is the $n\times n$ matrix $M$ associated with ${\sigma}$, whose $j,i$-entry is
\begin{equation}\label{mabc}
m_{ji}=\#\{k: {\sigma}(i)_k=j\},
\end{equation}
where $\#$ denotes the cardinality of the set. Notice that $m_{ji}$ is the number of times that the $T$-orbit of the interval
$J_i=(x'_{i-1},x_i')$ visits the interval $(x_{j-1},x_j)$ before return to intersect $I'$. In particular, we have that
\begin{equation}\label{ri}
r_i=\sum_{j=1}^n m_{ji}.
\end{equation} 
In what follows, we denote by $m^{(k)}_{ji}$ the $j,i$-entry of $M^k$. Moreover, $J_i$ and $r_i$ are as in the statement of Proposition \ref{277}.

 \begin{proposition}\label{xu} Let $T:I\to I$ be an $n$-IET self-similar on some admissible interval $I'\subset I$ in such a way that $(H3)$ holds. Given $k\ge 1$, let $J_i^{(k)}=L^{k-1}(J_i)$ for all $1\le i\le n$. Then 
\begin{equation}\label{floors}
\Big\{J_i^{(k)}, T\big(J_i^{(k)}\big),\ldots,T^{(r_i^{(k)}-1)}\big(J_i^{(k)}\big)\big\}, \quad 1\le i\le n,
\end{equation}
is an exhaustive family of $T$-towers, where 
 $r_i^{(k)}=\sum_{j=1}^n m^{(k)}_{ji}$.
  \end{proposition}
\begin{proof} By Corollary \ref{xyf}, 
we know that $\{J_i,T(J_i),\ldots, T^{r_i-1}(J_i)\}$, $1\le i\le n$, is an exhaustive family of $T$-towers. Hence, the result is true for  $k=1$ because $J_i^{(1)}=J_i$ and $r_i^{(1)}=r_i$. Since $T$ is self-similar on $I'$, we know that $T'$ is a rescaled copy of $T$. In particular,  by the above, $\big\{L(J_i),T'(L(J_i)),\ldots, {(T')}^{r_i-1}(L (J_i))\big\}$, that is,
 $\big\{J_i^{(2)},T'(J_i^{(2)}),\ldots, {(T')}^{r_i-1}(J_{i}^{(2)})\big\}$, $1\le i\le n$, is an exhaustive family of $T'$-towers. Translating this in terms of $T$, we obtain that  $\big\{J_i^{(2)},T(J_i^{(2)}),\ldots, T^{r_i^{(2)}-1}(J_{i}^{(2)})\big\}$, $1\le i\le n$, is an exhaustive family of $T$-towers, showing that the claim holds for $k=2$. Proceeding likewise, we prove that the claim is true for any $k\ge 1$.  \end{proof}
  
  \begin{corollary}\label{332} Let $T:I\to I$ be an $n$-IET self-similar on some admissible interval $I'\subset I$ in such a way that $(H3)$ holds. Suppose also that the following conditions are satisfied:
  \begin{itemize}
   \item[$(H4)$] \textrm{The matrix $M$ associated with $(T,I')$ is positive},
 \end{itemize}
  then $T$ is topologically transitive.
  \end{corollary}
  \begin{proof} Let $k\ge 1$ be given. For each $1\le i\le n$, let $J_i^{(k)}=L^{k-1}(J_i)$ be as in $(\ref{floors})$, where
  $L:I\to I'$ is the affine bijection $x\in I\mapsto \frac{a'}{a} x\in I'$. Let
  $$ \mathscr{P}_k=\big\{T^{\ell}(J_i^{(k)}): 0\le \ell\le r_i^{(k)}-1,\, 1\le i\le n\big\}.
  $$
  Then, by Proposition \ref{xu}, the union of the intervals in $\mathscr{P}_k$ is equal to $I$ up to finitely many points.  Moreover, by (H4), each interval
  $J_i^{(k+1)}$
   visits all the intervals in $\mathscr{P}_k$ before to return to intersect $\bigcup_{i=1}^n J_i^{(k+1)}$. Now let $U,V\subset I$ be open intervals. Since $\max_{J\in\mathscr{P}_k} \vert J\vert\to 0$ as $k\to \infty$, by taking $k$ large enough, we may assume that there exist intervals $J_U,J_V\in\mathscr{P}_{k}$ such that $J_U\subset U$ and $J_V\subset V$. Moreover, by the above, there exist $1\le i,j\le n$, $1\le \ell_U\le r_i^{(k+1)}$ and $1\le \ell_V\le r_j^{(k+1)}$ such that $T^{\ell_U}\big(J_i^{(k+1)}\big)\subset J_U\subset U$, $T^{\ell_V}\big(J_j^{(k+1)}\big)\subset J_V\subset V$
    and $T^{r_i^{(k+1)}}\big(J_i^{(k+1)}\big)\cap J_j^{(k+1)}$ is an open interval. In this way, there exists
    $k\ge 0$ such that $T^k(U)\cap V\neq\emptyset$. By Birkhoff's Transitivity Theorem, we have that $T$ has a dense orbit.
        \end{proof}

\section{The isometric model and the proof of Theorem \ref{thmT}}\label{sotr}

The aim of this section is to prove Theorem \ref{thmT}. The key step required to prove Theorem \ref{thmT} is showing that the map $T$ defined in $(\ref{formulaT})$ is topologically transitive. Unfortunately, we cannot apply Corollary \ref{332} directly to $T$ because $T$ is not self-similar. Thus, instead of $T$, we consider the Poincar\'e map $S=T'$ of $T$ on $I'=[0,1]$. More specifically, we will show that $I'$ is an admissible interval for $T$ and that (H3) holds true. Then, by Corollary \ref{48}, $T$ will be topologically transitive if so does $S$. This reduction is very convenient because, as we will show, $S$ is self-similar on the subinterval $\left[0,\frac{1}{\eta}\right]$ of $[0,1]$ and its topological transitivity will follow from Corollary \ref{332}. To conclude that $T$ is minimal we will prove that $T$ has no periodic orbit. These are the forthcoming steps.

In what follows,  let $T:[0,\vert\boldsymbol{\lambda}\vert]\to [0,\vert\boldsymbol{\lambda}\vert]$ be the map defined in
(\ref{formulaT}). Notice that $\mathcal{D}(T)=\{x_1,x_2,x_3\}$, where
$$ x_0=0,\quad x_1=\lambda_1,\quad x_2=\lambda_1+\lambda_2,\quad x_3=\lambda_1+\lambda_2+\lambda_3,\quad x_4=\lambda_1+\lambda_2+\lambda_3+\lambda_4=\vert\boldsymbol{\lambda}\vert.
$$

Some preparatory lemmas are necessary to prove Theorem \ref{thmT}.

 \subsection{Reduction Lemma}
 
\begin{lemma}\label{adi} $I'=[0,1]$ is an admissible interval for $T$. Moreover, the Poincar\'e map $T':I'\to I'$ is given by 
\begin{equation*}\label{SS'}
T'(x)=\begin{cases}
-x +\lambda_1+\lambda_3=-x-\nu_2+1 & \textrm{if}\quad x\in [x_{0}',x_1')\\
\phantom{-} x + \lambda_3  = x -\nu_1-\nu_2+1 & \textrm{if}\quad x\in [x_{1}',x_2']\\
\phantom{-} x +\lambda_2+\lambda_3-\vert\boldsymbol{\lambda}\vert = x-\nu_1-\nu_2 & \textrm{if}\quad x\in (x_{2}',x_3')\\
-x+\lambda_1+\lambda_2+\lambda_3=-x+\nu_3+1 & \textrm{if}\quad x\in [x_{3}',x_4']
\end{cases},
\end{equation*}
where $$ x_0'=0,\quad x_1'=\nu_1, \quad x_2'=\nu_1+\nu_2, \quad x_3'=\nu_1+\nu_2+\nu_3,\quad x_4'=1,
$$
and $\mathcal{D}(T')=\{x_1',x_2',x_3'\}$.

\end{lemma}
\begin{proof} See the Appendix.
\end{proof}

\begin{lemma}[Reduction Lemma]\label{rl2} If $T'$ is topologically transitive, then so is $T$.
\end{lemma}
\begin{proof} See the Appendix.
 \end{proof}
 
 \subsection{The map $S$}\label{subS}\rule[-5pt]{0pt}{5pt}\\
 
  Let
 $S:[0,1]\to [0,1]$ be the $4$-IET defined by
 
 \begin{figure}[!htb]\vspace{-0.6cm}
   \begin{minipage}{0.55\textwidth}
     
     \noindent  \\[0.0in]
    $\phantom{}S(x)=\begin{cases}
-x - \nu_2 + 1 & \textrm{if}\quad x\in [y_0,y_1)\\
\phantom{-}x-\nu_1-\nu_2+1 & \textrm{if}\quad x\in [y_1,y_2]\\
\phantom{-}x-\nu_1-\nu_2 & \textrm{if}\quad x\in (y_2,y_3)\\
-x+\nu_3+1 & \textrm{if}\quad x\in [y_3,y_4]
\end{cases},$
\\ \\
   \end{minipage}\hfill
   \begin {minipage}{0.45\textwidth}
     \centering
 \begin{tikzpicture}[scale=0.7]



\draw [  thick, ->] (0,0) -- (5.5,0) node [right] {\footnotesize $x$};
\draw [  thick, ->] (0,0) -- (0,5.5) node [below left] {\footnotesize $S(x)$};
	
			

\draw (0,0)--(5,0)--(5,5)--(0,5);
			
\draw[fill=black] (0,3.98026) circle (0.1);
\draw[fill=white] (1.72223, 2.25803) circle (0.1);
\draw[very thick] (0,3.98026)--(1.65, 2.33);

\draw[fill=black] (1.72223,3.98026) circle (0.1);
\draw[fill=black] (2.74197, 5) circle (0.1);
\draw[very thick] (1.72223,3.98026)--(2.67, 4.92);

\draw[fill=white] (2.74197,0) circle (0.1);
\draw[fill=white] (3.27777, 0.535799) circle (0.1);
\draw[very thick] (2.7897,0.05)--(3.22, 0.467);

\draw[fill=black] (3.27777,2.25803) circle (0.1);
\draw[fill=black] (5,0.535799) circle (0.1);
\draw[very thick] (3.27777,2.25803)--(5,0.535799);

\draw [  thick, dotted] (1.72223,0)--(1.72223,5);
\draw [  thick, dotted] (2.74197,0)--(2.74197,5);
\draw [  thick, dotted] (3.27777,0)--(3.27777,5);

\draw [  thick, dotted] (0,0.535799)--(5,0.535799);
\draw [  thick, dotted] (0,2.25803)--(5,2.25803);
\draw [  thick, dotted] (0,3.98026)--(5,3.98026);

\draw (0.861115,0) node[below] {};
 \draw [thick, black,decorate,decoration={brace,amplitude=5pt,mirror},xshift=0pt,yshift=-0.8pt](0,-0.2) -- (1.72223,-0.2) node[black,midway,yshift=-0.4cm] {\footnotesize $\nu_1$};
 \draw[thick, dotted] (0,0)--(0,-0.2);
 \draw[thick, dotted] (1.72223,0)--(1.72223,-0.2);
 
 \draw (2.2321,0) node[below] {};
  \draw [thick, black,decorate,decoration={brace,amplitude=5pt,mirror},xshift=0pt,yshift=-0.8pt](1.72223,-0.2)--(2.74197,-0.2) node[black,midway,yshift=-0.4cm] {\footnotesize $\nu_2$};
   \draw[thick, dotted] (2.74197,0)--(2.74197,-0.2);
  
   \draw (3.00987,0) node[below] {};
    \draw [thick, black,decorate,decoration={brace,amplitude=4pt,mirror},xshift=0pt,yshift=-0.8pt] (2.74197,-0.2) --(3.27777,-0.2) node[black,midway,yshift=-0.4cm] {\footnotesize $\nu_3$};
     \draw[thick, dotted] (3.27777,0)--(3.27777,-0.2);
    
    \draw (4.13888,0) node[below] {};
     \draw [thick, black,decorate,decoration={brace,amplitude=5pt,mirror},xshift=0pt,yshift=-0.8pt] (3.27777,-0.2)--(5,-0.2) node[black,midway,yshift=-0.4cm] {\footnotesize $\nu_4$};
      \draw[thick, dotted] (5,0)--(5,-0.2);

\end{tikzpicture}
   \end{minipage}
\end{figure}\vspace{-1cm}

 where
\begin{equation}\label{theys}
 y_0=x_0'=0,\quad y_1=x_1'=\nu_1, \quad y_2=x_2'=\nu_1+\nu_2, \quad y_3=x_3'=\nu_1+\nu_2+u_3,\quad y_4=x_4'=1.
\end{equation}
 Then $\mathcal{D}(S)=\{y_1,y_2,y_3\}$. In the previous subsection, we proved that $S=T'$. Let $L:[0,1]\to \left[0,\frac{1}{\eta} \right]$ be the map $L(y)=\frac{1}{\eta}y$. Set $y_i'=L(y_i)$, $1\le i\le 4$, then
 $$ y_0'=0, \quad y_1'=\frac{1}{\eta} y_1,\quad y_2'=\frac{1}{\eta} y_2,\quad y_3'=\frac{1}{\eta} y_3,\quad y_4'=\frac{1}{\eta}.
 $$
 The proofs of the next three lemmas are given in the Appendix.
 
 \begin{lemma}\label{aab} $\left[0,\frac{1}{\eta}\right]$ is an admissible interval for $S$.
  \end{lemma}
  \begin{proof} See the Appendix.
  \end{proof}

\begin{lemma}\label{aac} $S$ is self-similar on $\left[0,\frac{1}{\eta}\right]$.
\end{lemma}
\begin{proof} See the Appendix.
  \end{proof}

\begin{lemma}\label{Sistt} $S$ is topologically transitive.
\end{lemma}
\begin{proof} See the Appendix.
\end{proof}

\begin{lemma}\label{Tistt} $T$ is topologically transitive. 
\end{lemma}
\begin{proof} By Lemma \ref{Sistt}, $S$ is topologically transitive. Since $S=T'$, we. have that $T'$ is also topologically transitive. The proof is concluded by applying Lemma \ref{rl2}.
\end{proof}

\begin{proof}[Proof of Theorem \ref{thmT}] The topological dynamics of $n$-IETs is well-understood. In particular, it is known that the domain of $T$ splits into the union of periodic components, minimal components and $T$-connections (see \cite[Theorem 3.2]{ANBPST2013} and \cite[pp. 470-480]{AKBH1995}). By Lemma \ref{Tistt}, $T$ is topologically transitive, thus $T$ has no periodic component and  has a unique minimal component. Moreover, the minimal component is also a quasi-minimal set in the sense that every non-periodic orbit is dense in it. In this way, $T$ will be minimal if we show that $T$ has no periodic orbit. By way of contradiction, suppose that $T$ has a periodic orbit $\gamma$. Then $\gamma$ contains at least one discontinuity of $T$, otherwise there would exist a periodic component containing $\gamma$. In particular, $T$ has a $T$-connection, that is, there exist $k\ge 1$ and $x_i,x_j \in \mathcal{D}(T)$ such that
$T^k(x_i)=x_j$ and $T^\ell(x_i)\not\in \mathcal{D}(T)$ for all $0<\ell<k$. This contradicts the fact that the Poincar\'e map of $T$ on $I'$ is a self-similar $4$-IET. Therefore, $T$ has no periodic orbit, showing that $T$ is minimal.

Now let us prove that $T$ is uniquely ergodic. Since $T$ has no periodic orbit, all the $T$-invariant measures are non-atomic and are supported on an uncountable set. Let $\mu_1,\mu_2$ be two (non-atomic) $T$-invariant Borel probability measures, then $\mu_1'=\frac{1}{\mu_1([0,1])}\mu_1$ and $\mu_2'=\frac{1}{\mu_2([0,1])}\mu_2$ are $S$-invariant Borel probability measures. Moreover, by the proof of Lemma \ref{rl2}, $T$ satisfies (H3) on $[0,1]$, then $\mu_1=\mu_2$ if and only if
$\mu_1'=\mu_2'$. Since $S$ is self-similar on $\left[0,\frac1\eta\right]$, we have that any $S$-invariant Borel probability measure $\mu'$ is determined by the vector $\mathbf{r}=\Big(\mu'\big((y_0,y_1)\big), \mu'(\big(y_1,y_2)\big), \mu'\big((y_2,y_3)\big), \mu'\big((y_3,y_4)\big)\Big)$ which has strictly positive entries, where $y_0,\ldots,y_4$ are as in (\ref{theys}).
 Moreover, since $S$ is self-similar, we have that $\boldsymbol{\nu}$ is the only probability eigenvector of $P$ with strictly positive entries, that is,
$\mathbf{r}=\boldsymbol{\nu}$. This means that the only $S$-invariant measure is the Lebesgue measure, then
$\mu_1'=\mu_2'$ and so $\mu_1=\mu_2$. This proves that $T$ is uniquely ergodic.  \end{proof}

\section{Piecewise contractions and the proof of Theorem \ref{omr}}

In this section, we will prove Theorem \ref{omr}. By Lemma \ref{lmint} and by Theorem \ref{thmT}, all we have to do is to find parameters $d_1,d_2,d_3>0$ such that the map $f_{d_1,d_2,d_3}$ defined in (\ref{fd1d2d3}) is topologically semiconjugate to $T$. The map $f_{d_1,d_2,d_3}$ is a piecewise $\frac12$-affine contraction in the following sense.
\begin{definition}[{piecewise $\frac12$-affine contraction}] A map $f:[0,1]\to [0,1]$ is a \textit{piecewise $\frac12$-affine contraction} if
there exist a partition of $[0,1]$ into intervals $J_1,\ldots,J_n$, numbers $a_1,\ldots,a_n\in\left\{-\frac12,\frac12\right\}$ and $b_1,\ldots b_n\in\mathbb{R}$
 such that $f(x)=a_i x+b_i$ for all $x\in J_i$ $(i=1,2,\ldots,n)$. 
\end{definition}

Our strategy is the following: first we construct a class $\mathscr{C}$ of piecewise $\frac12$-affine contractions topologically semiconjugate to $T$ (Proposition \ref{pro1}). Then we prove there exist $d_1,d_2,d_3>0$ such that
 $f_{d_1,d_2,d_3}\in\mathscr{C}$ (Proposition \ref{pro239}).
 \begin{definition}[The map $g_{\mathbf{u},\ell}$]\label{gredef}
Given vectors $\boldsymbol{u}=(u_1,u_2,u_3,u_4)$ and  $\boldsymbol{\ell}=(\ell_1,\ell_2,\ell_3,\ell_4)$ with positive entries satisfying $\vert\boldsymbol{u}\vert=u_1+u_2+u_3+u_4=1$ and
 $\vert\boldsymbol{\ell}\vert=\ell_1+\ell_2+\ell_3+\ell_4=\frac12$, let $g_{\boldsymbol{u},\boldsymbol{\ell}}:[0,1]\to [0,1]$ be the  piecewise $\frac12$-affine contraction defined by 
\begin{equation}\label{gul}
g_{\boldsymbol{u},\boldsymbol{\ell}}(x)=\begin{cases} -\dfrac{x}{2} +\dfrac{u_1}{2} + \dfrac{u_3}{2} + \ell_1+\ell_2 & \textrm{if} \quad  x\in J_1\\[0.15in]
-\dfrac{x}{2} +\dfrac{u_1}{2} + \dfrac{1}{2}+\ell_1+\ell_2+\ell_3 & \textrm{if} \quad x\in J_2 \\[0.15in]
-\dfrac{x}{2}+\dfrac{u_1}{2}+\dfrac{u_2}{2} + \dfrac{u_3}{2} +\ell_1 & \textrm{if}\quad x\in J_3 \\[0.15in]
-\dfrac{x}{2} + \dfrac{u_1}{2} + \dfrac{u_3}{2} + \dfrac{1}{2}+\ell_1+\ell_2& \textrm{if} \quad x\in J_4
\end{cases},
\end{equation}
where $J_1,J_2,J_3,J_4$ is the partition of $[0,1]$ given by 
$$ J_1=[0,u_1),\quad J_2=[u_1,u_1+u_2),\quad J_3=[u_1+u_2,u_1+u_2+u_3),\quad J_4=[u_1+u_2+u_3,1].
$$
Also let 
$$\mathscr{C}=\left\{g_{\boldsymbol{u},\boldsymbol{\ell}}: u_i,\ell_i>0,\forall i, \sum_{i=1}^4 u_i=1\,\,\textrm{and}\,\,\sum_{i=1}^4\ell_i=\frac12 \right\}.$$
\end{definition}
In what follows, let $T:[0,\vert\boldsymbol{\lambda}\vert]\to [0,\vert\boldsymbol{\lambda}\vert]$ be the isometric model and let $I_1,I_2, I_3, I_4$ be the partition of $[0,\vert\boldsymbol{\lambda}\vert]$ associated with $T$ (see (\ref{formulaT})). We will also keep all the notations and values given in Sections $1$ and $2$. Let
$$p_1=0, \quad p_2=T(\lambda_1+\lambda_2), \quad  p_3=T(\lambda_1+\lambda_2+\lambda_3) \quad\textrm{and}\quad  p_4=\vert\boldsymbol{\lambda}\vert.
$$ 
\begin{lemma}\label{leminv} The $T$-orbits of $p_1,p_2$ and $p_3$ are pairwise disjoint.
\end{lemma}
\begin{proof} Denote by $O(x)=\{x,T(x),\ldots\}$  the $T$-orbit of $x\in [0,\vert\boldsymbol{\lambda}\vert]$. By (\ref{formulaT}), $T(\lambda_1)=\vert\boldsymbol{\lambda}\vert$ and $T(0)=T(\vert\boldsymbol{\lambda}\vert)$. Hence, 
$$ O(p_1)\subset \{0\}\cup O(\lambda_1), \quad O(p_2)\subset O(\lambda_1+\lambda_2), \quad O(p_3)\subset O(\lambda_1+\lambda_2+\lambda_3).
$$
In the proof of Theorem \ref{thmT}, we showed that $T$ has no $T$-connection, thus there exists no $T$-orbit that passes through two discontinuities of $T$. This together with the injectivity of $T$ on $(0,\vert\boldsymbol{\lambda}\vert]$ implies that $O(\lambda_1)$, $O(\lambda_1+\lambda_2)$ and $O(\lambda_1+\lambda_2+\lambda_3)$ are pairwise disjoint. Moreover, we have that $0$ has no preimage, which concludes the proof.
  \end{proof}

\begin{proposition}\label{pro1} Let $\boldsymbol{u}=(u_1,u_2,u_3,u_4)$  and $\boldsymbol{\ell}=(\ell_1,\ell_2,\ell_3,\ell_4)$ be vectors with positive entries
 satisfying $\sum_{i=1}^4 u_i=1$ , $\sum_{i=1}^4 \ell_i=\frac12$, and \begin{equation}\label{324}
 \begin{pmatrix} 
u_1 \\ u_2 \\ u_3
\end{pmatrix}=
M
\begin{pmatrix}
 \ell_1 \\ \ell_2 \\ \ell_3
\end{pmatrix}+\frac12\begin{pmatrix}
 c_{14} \\ c_{24} \\ c_{34}
\end{pmatrix},
\end{equation}
then $g=g_{{\boldsymbol{u}},\boldsymbol{\ell}}$ is topologically semiconjugate to $T$. \end{proposition}
\begin{proof} Let $\boldsymbol{\ell}=(\ell_1,\ell_2,\ell_3,\ell_4)$ be a vector with positive entries such that
$\sum_{i=1}^4 \ell_i=\frac12$. Let $$\mathcal{P}=\left\{T^k(p_i): k\ge 0 \,\,\textrm{and}\,\, 1\le i\le 4\right\}.$$ By Theorem \ref{thmT}, $\mathcal{P}$ is a denumerable dense subset of $[0,\vert\boldsymbol{\lambda}\vert]$. 
Since $T(p_1)=T(p_4)$, we may write
 $\mathcal{P}=\left\{T^k(p_i):k\ge 0\,\,\textrm{and}\,\,1\le i\le 3\right\}\cup \{p_4\}$. Let $\phi:\mathcal{P}\to (0,1)$ be the map defined by $\phi(p_i)=\ell_i$, $1\le i\le 4$, and, for all $k\ge 1$,
$$
\phi\left(T^k(p_1)\right)=\dfrac{\ell_1+\ell_4}{2^k}, \,\,\,\,\phi\left(T^k(p_2)\right)=\dfrac{\ell_2}{2^k},\,\,\,\,\phi\left(T^k(p_3)\right)=\dfrac{\ell_3}{2^k}.
$$
By Lemma \ref{leminv}, $\phi$ is well-defined.  To each $p\in\mathcal{P}$, let $G_p\subset [0,1]$ be the compact interval defined by $G_{p_1}=\left[0,{\ell_1}\right], G_{p_4}=\left[1-{\ell_4}, 1\right]$ and
\begin{equation}\label{GP}
G_p=\left[\sum_{\substack{q<p\\q\in \mathcal{P}}} \phi(q),\,\,\phi(p)+\sum_{\substack{q<p\\q\in \mathcal{P}}} \phi(q)\right]\,\,\,\,\textrm{if}\,\,\,\, p\not\in \{p_1,p_4\}.
\end{equation}
Notice that $G_p$ has length $\vert G_p\vert=\phi(p)$ for all $p\in\mathcal{P}$.
 Hence,
\begin{equation}\label{UGP}
\sum_{p\in \mathcal{P}} \left\vert G_p\right\vert=\sum_{i=1}^4 \ell_i+
 \sum_{k\ge 1} \dfrac{\ell_1+\ell_4+\ell_2+\ell_3}{2^k}=\frac12\left(1+\sum_{k\ge 1} \dfrac{1}{2^k}\right)=1.
\end{equation}
By $(\ref{GP})$ and by the density of $\mathcal{P}$ in $[0,\vert\boldsymbol{\lambda}\vert]$,
we have that $\mathcal{P}$ and $\{G_p\}_{p\in \mathcal{P}}$ share the same ordering meaning that if $p,q\in \mathcal{P}$, then
\begin{equation}\label{p<q}
 p<q \iff \sup G_p < \inf G_q.
\end{equation}
In particular, we have that the intervals $G_p$, $p\in\mathcal{P}$, are pairwise disjoint and, by $(\ref{UGP})$, their union is dense in $[0,1]$.
 
Let $\widehat{h}:\,\cup_{p\in\mathcal{P}} G_p\to [0,\vert\boldsymbol{\lambda}\vert]$ be the function that on $G_p$ takes the constant value $p$. By (\ref{UGP}) and $(\ref{p<q})$, we have that $\widehat{h}$ is nondecreasing and has  dense domain and dense range. Thus, $\widehat{h}$ admits a unique nondecreasing continuous surjective extension $h:\,[0,1]\to [0,\vert\boldsymbol{\lambda}\vert]$. It is elementary to see that
 $h^{-1}\big(\{p\}\big)={G_p}$ for every $p\in\mathcal{P}$. Denote by $J_1, J_2, J_3, J_4$ the partition of $[0,1]$ defined by
 $J_i=h^{-1}(I_i)$, where $I_1,I_2, I_3, I_4$ are as in the definition of the isometric model $T$.
 
 Let $\widehat{g}:\cup_{p\in \mathcal{P}} G_p\to \cup_{p\in\mathcal{P}} G_{T(p)}$ be such that $\widehat{g}\vert_{G_p}{:}\,G_p\to G_{T(p)}$ is an  affine bijection with slope $-\frac12$ for every $p\in\mathcal{P}$. 
 We claim that for each $1\le i\le 4$ , there exist a dense subset $\widehat{J}_i$ of $J_i$ and $b_i\in\mathbb{R}$ such that
\begin{equation}\label{lxb}
\widehat{g}(x)=-\frac12 x+b_i\quad\textrm{for all}\quad x\in {\widehat J}_i\,.
\end{equation}
 Let $1\le i\le 4$, $\widehat{I_i}=I_i\cap \mathcal{P}$, and  $\widehat{J_i}={\cup_{p\in \widehat{I_i}} G_{p}}$, then, by (\ref{UGP}) and (\ref{p<q}), $\widehat{J_i}$ is a dense subset of $J_i$. Moreover, by definition, $\widehat{g}\vert_{G_{p}}{:}\,G_{p}\to G_{T(p)}$ is an  affine bijection with slope $-\frac12$ for all $p\in \mathcal{P}$, thus there exists $c_{p}\in\mathbb{R}$ such that
  \begin{equation}\label{cmkcmk}
   \widehat{g}(x)=-\frac{1}{2} x + c_{p}\quad\textrm{for all}\quad x\in G_{p}\,\,\textrm{and}\,\, p\in\mathcal{P}.
  \end{equation}
Let us prove that $\widehat{g}$ is strictly decreasing on $\widehat{J_i}=\cup_{p\in\widehat{I_i}} G_{p}$. Let  $x<y$ be two points in $\widehat{J_i}$. Since $\widehat{g}$ is already strictly decreasing on each interval $G_{p}$, we may assume that $x\in G_{p}$ and $y\in G_{q}$, where $p,q\in \widehat{I_i}$ are such that $\sup G_p<\inf G_{q}$. By $(\ref{p<q})$, we have that  
  $p<q$ and $\{p,q\}\subset I_i$. Then, since $T'(z)=-1$ for all $z\in I_i$, we have that $T\vert_{I_i}$ is decreasing, thus  $T(p)>T(q)$. By $(\ref{p<q})$ once more, we get $\sup G_{T(q)}<\inf G_{T(p)}$.
   By definition, $\widehat{g}(p)\in G_{T(p)}$ and $\widehat{g}(q)\in G_{T(q)}$, thus $\widehat{g}(p)>\widehat{g}(q)$. This proves that $\widehat{g}$ is decreasing on $\widehat{J_i}$.  It remains to prove that $c_{p}$ in $(\ref{cmkcmk})$ is the same for all $p\in \widehat{I_i}$. Let $p,q\in\widehat{I_i}$ with $p\neq q$. We may assume that $a=\sup {G_{p}}<\inf G_{q}=b$. Notice that since $\widehat{g}$ is decreasing on $\widehat{J_i}$,
   \begin{eqnarray*} 
   \frac12(b-a)-(c_{q}-c_{p})&=&
   -\big(\widehat{g}(b)-\widehat{g}(a)\big)= \sum_{G_{r}\subset [a,b]} \left\vert \widehat{g}\big(G_{r}\big)\right\vert \\ &=&\frac12 \sum_{G_{r}\subset [a,b]} |G_{r}|= \frac12 (b-a),
   \end{eqnarray*}
  yielding $c_{p}=c_{q}$. Thus, $(\ref{lxb})$ is true.
      
  It follows from $(\ref{lxb})$ that $\widehat{g}\vert_{\widehat{J_i}}$ admits a unique monotone continuous extension to the interval $J_i=h^{-1}(I_i)$. This extension is also an affine map with slope equal to $-\frac12$. Since $i$ is arbitrary, we obtain an injective  piecewise $\frac12$-affine extension $g$ of $\widehat{g}$ to the whole interval $[0,1]=\cup_{i=1}^4 J_i$. 
        
  We claim that $h\circ g=T\circ h$. In fact, for every $y\in G_p$, we have that 
\begin{equation}\label{eql}
h\big(g(y)\big)=\widehat{h}\big(\widehat{g}(y)\big)=T(p)=T\big(\widehat{h}(y)\big)=T\big(h(y)\big).
\end{equation}
Hence, $(\ref{eql})$ holds for a dense set of $y\in [0,1]$. By continuity, $(\ref{eql})$ holds for every $y\in [0,1]$. In this way, $g$ is topologically semiconjugate to $T$.

Figure \ref{fig20} gives a geometrical picture of the map $g$.  
All the slopes equal $-\frac{1}{2}$. It is elementary to verify that $g=g_{\mathbf{u},\boldsymbol{\ell}}$, where
 $u_i=\vert J_i\vert$. Thus the formula of $g$ is the one provided in Definition \ref{gredef}.

\begin{figure}[ht]. 
\centering
\begin{tikzpicture}[scale=0.9]



\draw [  thick, ->] (0,0) -- (7,0) node [right] {\footnotesize $x$};
\draw [  thick, ->] (0,0) -- (0,5.5) node [above] {\footnotesize $g_{\boldsymbol{u},\boldsymbol{\ell}}(x)$};
	
			

\draw (0,0)--(5,0)--(5,5)--(0,5);
			
\draw[fill=black] (0,2.45) circle (0.1);
\draw[fill=white] (1.25, 1.825) circle (0.1);
\draw[very thick] (0,2.45)--(1.18, 1.87);

\draw[fill=black] (1.25,4.375) circle (0.1);
\draw[fill=white] (2.5, 3.68643) circle (0.1);
\draw[very thick] (1.33,4.33)--(2.41, 3.75);

\draw[fill=black] (2.5,1.25) circle (0.1);
\draw[fill=white] (3.75, 0.625) circle (0.1);
\draw[very thick] (2.57,1.21)--(3.67, 0.675);

\draw[fill=black] (3.75,3.125) circle (0.1);
\draw[fill=black] (5,2.45) circle (0.1);
\draw[very thick] (3.82,3.09)--(4.9,2.49);

\draw [  thick, dotted] (1.25,0)--(1.25,4.9);
\draw [  thick, dotted] (2.5,0)--(2.5,5);
\draw [  thick, dotted] (3.75,0)--(3.75,5);

\draw [  thick, dotted] (0,1.825)--(5,1.825);
\draw [  thick, dotted] (0,2.45)--(5,2.45);

\draw (0.625,0) node[below] {$J_1$};
 \draw [thick, black,decorate,decoration={brace,amplitude=5pt,mirror},xshift=0pt,yshift=-0.8pt](0,-0.7) -- (1.25,-0.7) node[black,midway,yshift=-0.4cm] {\footnotesize $u_1$};
 \draw[thick, dotted] (0,0)--(0,-0.7);
 \draw[thick, dotted] (1.25,0)--(1.25,-0.7);
 
 \draw (5,0.30) node[right] {$G_{p_1}$};
 \draw [thick, black,decorate,decoration={brace,amplitude=5pt,mirror},xshift=0pt,yshift=0pt](6,0) -- (6,0.625) node[black,midway,xshift=0.4cm] {\footnotesize $\ell_1$};
 \draw[thick, dotted] (5,0)--(6,0);
 \draw[thick, dotted] (0,0.625)--(6,0.625);
 
 \draw (1.875,0) node[below] {$J_2$};
  \draw [thick, black,decorate,decoration={brace,amplitude=5pt,mirror},xshift=0pt,yshift=-0.8pt](1.25,-0.7)--(2.5,-0.7) node[black,midway,yshift=-0.4cm] {\footnotesize $u_2$};
   \draw[thick, dotted] (2.5,0)--(2.5,-0.7);
   
    \draw (5,1.56) node[right] {$G_{p_2}$};
 \draw [thick, black,decorate,decoration={brace,amplitude=5pt,mirror},xshift=0pt,yshift=0pt](6,1.25) -- (6,1.825) node[black,midway,xshift=0.4cm] {\footnotesize $\ell_2$};
 \draw[thick, dotted] (0,1.25)--(6,1.25);
 \draw[thick, dotted] (5,1.825)--(6,1.825);
  
   \draw (3.125,0) node[below] {$J_3$};
    \draw [thick, black,decorate,decoration={brace,amplitude=5pt,mirror},xshift=0pt,yshift=-0.8pt] (2.5,-0.7) --(3.75,-0.7) node[black,midway,yshift=-0.4cm] {\footnotesize $u_3$};
     \draw[thick, dotted] (3.75,0)--(3.75,-0.7);
     
     \draw (5,3.4375) node[right] {$G_{p_3}$};
 \draw [thick, black,decorate,decoration={brace,amplitude=5pt,mirror},xshift=0pt,yshift=0pt](6,3.125) -- (6,3.75) node[black,midway,xshift=0.4cm] {\footnotesize $\ell_3$};
 \draw[thick, dotted] (0,3.125)--(6,3.125);
 \draw[thick, dotted] (0,3.75)--(6,3.75);
    
    \draw (4.375,0) node[below] {$J_4$};
     \draw [thick, black,decorate,decoration={brace,amplitude=5pt,mirror},xshift=0pt,yshift=-0.8pt] (3.75,-0.7)--(5,-0.7) node[black,midway,yshift=-0.4cm] {\footnotesize $u_4$};
      \draw[thick, dotted] (5,0)--(5,-0.7);
     
      \draw (5,4.6875) node[right] {$G_{p_4}$};
 \draw [thick, black,decorate,decoration={brace,amplitude=5pt,mirror},xshift=0pt,yshift=0pt](6,4.375) -- (6,5) node[black,midway,xshift=0.4cm] {\footnotesize $\ell_4$};
 \draw[thick, dotted] (0,4.375)--(6,4.375);
 \draw[thick, dotted] (5,5)--(6,5);
\end{tikzpicture}
\caption{The plot of $g=g_{\boldsymbol{u},\boldsymbol{\ell}}$.}
\label{fig20}
\end{figure}
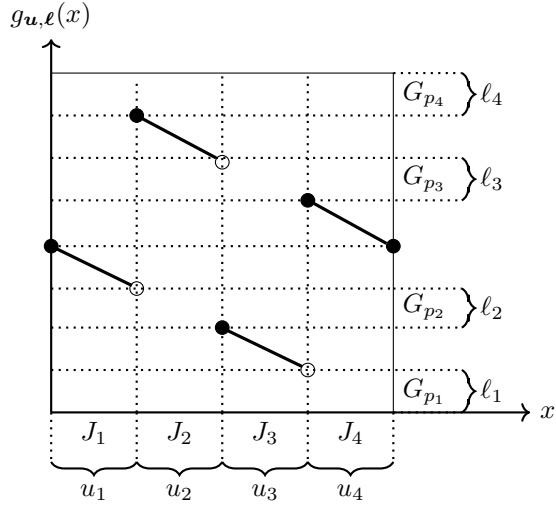
It remains to prove that $\mathbf{u}=(u_1,u_2,u_3,u_4)$ satisfies (\ref{324}). In fact,
$\sum_{i=1}^4 u_i=\sum_{i=1}^4 \vert J_i\vert=1$. Moreover,  we have that
\begin{equation*}
 u_i=\vert J_i\vert=\sum_{G_p\subset J_i}\left\vert G_p\right\vert =\sum_{p\in \mathcal{P}\cap I_i} \phi(p)
 =\sum_{j=1}^4\sum_{k\in K_{ij}} \dfrac{\ell_j}{2^k}=\sum_{j=1}^4 c_{ij} \ell_j.
   \end{equation*}
    Replacing $\ell_4$ by $\frac{1}{2}-\ell_1-\ell_2-\ell_3$ yields, for all $1\le i\le 3$,
 \begin{equation*}
 u_i= \sum_{j=1}^3 (c_{ij} -c_{i4}) \ell_j+\frac12 c_{i4},
 \end{equation*}
 which concludes the proof. 
\end{proof}

\begin{proposition}\label{pro239} Let $\boldsymbol{u}=(u_1,u_2,u_3,u_4)$ be such that
$u_1,u_2,u_3>0$, $u_4=1-u_1-u_2-u_3$,
\begin{equation}\label{vvv}
0< u_1 < \dfrac13, \quad\dfrac13< u_1+u_2 < \dfrac23,\quad \dfrac23<u_1+u_2+u_3<1,
\end{equation}
and let $\boldsymbol{\ell}=(\ell_1,\ell_2,\ell_3,\ell_4)$ be a vector with positive entries
 satisfying $\sum_{i=1}^4  \ell_i=\frac12$. If 
 \begin{equation}\label{M201}
 \begin{pmatrix} 
2 & 2 & 0\\
2 & 2 & 2\\
2 & 0 & 0 \\
\end{pmatrix}
\begin{pmatrix} 
\ell_1 \\ \ell_2 \\ \ell_3
\end{pmatrix}=
\begin{pmatrix} 
-1 & 0 & -1\\
-1 & 0 & 0\\
-1 & -1 & -1 \\
\end{pmatrix}
\begin{pmatrix}
 u_1 \\ u_2 \\ u_3 
\end{pmatrix}+\begin{pmatrix}
1 \\ 1 \\ 1
\end{pmatrix},
\end{equation}
 \begin{equation}\label{y201}
 z_1=u_1,\quad z_2=u_1+u_2, \quad z_3=u_1+u_2+u_3,
 \end{equation}
 and 
 \begin{equation}\label{d201}
 d_1=\frac{1}{3z_1}-1, \quad d_2=\dfrac{2-3z_2}{3z_2-1},\quad d_3=\dfrac{3-3z_3}{3z_3-2},
 \end{equation}
then $g_{{\boldsymbol{u}},\boldsymbol{\ell}}=f_{d_1,d_2,d_3}$, that is, $g_{{\boldsymbol{u}},\boldsymbol{\ell}}$ is the Poincar\'e map of a switched server system.
\end{proposition}
\begin{proof} By replacing (\ref{M201}) in (\ref{gul}), and (\ref{y201}) and (\ref{d201}) in  (\ref{fd1d2d3}), it can be easily verified that  $g_{{\boldsymbol{u}},\boldsymbol{\ell}}=f_{d_1,d_2,d_3}$.
\end{proof}
\begin{proof}[Proof of Theorem \ref{omr}] The itens (a) and (b) of Theorem \ref{omr} follow immediately from Propositions \ref{pro1}, \ref{pro239} and Theorem \ref{thmT}. Let $\mathbf{v}\in\partial\Delta$.  
It is clear that $\omega_F(\mathbf{v})$ is a closed, therefore compact, set for every $\mathbf{v}\in\partial\Delta$. Since $F:\partial\Delta\to\partial\Delta$ is topologically conjugate to a piecewise contraction $f:[0,1]\to [0,1]$ injective on $(0,1]$, we have that $\omega_F(\mathbf{v})$ has empty interior, hence $\omega_F(\mathbf{v})$ is totally disconnected. Since, by the item (b), $F$ is topologically semiconjugate to $T$, we have that $\omega_F(\mathbf{v})$ is a perfect set. In this way, $\omega_F(\mathbf{v})$ is a Cantor set. This proves the item (c). Let us prove the item (d). Let $0\le t_1<t_2\cdots$ be the switching times. Let $\mathbf{v}(t_k)=
\big(v_1(t_k),v_2(t_k),v_3(t_k)\big)$ be the state of the server at the time $t_k$. By (\ref{iff2-38}), we have that  $l(t_k)=1$ (i.e. the server is connected to the tank $1$) if and only if $\mathbf{v}(t_k)\in  [\mathbf{r}_2,\mathbf{e}_1] \cup [\mathbf{e}_1,\mathbf{r}_3)$.  Since $F$ is topologically semiconjugate to $T$, this translates into interval dynamics as follows: $l(t_k)=1$ if and only if $w_k\in [\lambda_1,\lambda_1+\lambda_2+\lambda_3)$, where $w_k=h(\mathbf{v}(t_k))$ is the projection of $\mathbf{v}(t_k)$ by the topological semiconjugacy $h$. In this way, since $T$ is uniquely ergodic, the normalized Lebesgue measure $\mu$ in the only $T$-invariant Borel probability measure, then by the version of Birkhoff's Ergodic Theorem for uniquely ergodic transformations (see \cite[Proposition 4.1.13]{AK1980}), we reach
\begin{eqnarray*}
\textrm{freq}\,(1) &=&\lim_{n\to\infty}\frac1n\#\{1\le k\le n:l(t_k)=1\}\\&=&\lim_{n\to\infty}\frac1n\#\{1\le k\le n: w_k\in [\lambda_1,\lambda_1+\lambda_2+\lambda_3)\}\\&=&\lim_{n\to\infty}\frac1n\#\{1\le k\le n: T^{k-1}(w_1)\in [\lambda_1,\lambda_1+\lambda_2+\lambda_3)\}
\\ &=&\mu\big([\lambda_1,\lambda_1+\lambda_2+\lambda_3)\big)=\frac{\lambda_3}{\vert\boldsymbol{\lambda}\vert}.
\end{eqnarray*}
Proceeding likewise, one can prove that $\textrm{freq}\,(2)=\dfrac{\lambda_1+\lambda_4}{\vert\boldsymbol{\lambda}\vert}$ and $
\textrm{freq}\,(3)=\dfrac{\lambda_2}{\vert\boldsymbol{\lambda}\vert}$.

\end{proof}

\section*{Appendix}\label{app}

The proofs of the results that demand numerical analysis are provided in this section. In order not to overstretch the discussion,
we skip some details. Since the isometric model $T:I\to I$ is a piecewise-defined map, in order to compute $T^k(x)$, it is necessary to know which of the intervals $I_1,I_2,I_3,I_4$ the point $T^{k-1}(x)$ belongs to. In other words, we need to know the address $i_{k-1}$ determined by the equation $T^{k-1}(x)\in I_{i_{k-1}}$. By recursion, if we know the word
$i_0 i_1 \ldots i_{k-1}$, then we can compute $T^k(x)$ exactly by means of Corollary \ref{xy17}. All we need is to compute
$\{x,T(x),\ldots,T^{k}(x)\}$ for finitely many $x$'s and finitely many $k$'s.

\subsection*{Spectral analysis of the matrix $P$}\rule[-5pt]{0pt}{5pt}\\ 

\indent The characteristic polynomial $p$ of $P$ is the product of polynomials:
$$ p(t)=(t-1)(t^3-11 t^2 + 7 t -1).
$$
Hence, the Perron-Frobenius eigenvalue $\eta$ of $P$ is a root of the irreducible polynomial over $\mathbb{Q}$: $t^3-11 t^2 + 7 t -1$. In particular, $1,\eta$ and $\eta^2$ are rationally independent. Namely, $\eta$ is equal to
$$ \eta=\frac{1}{3} \left(11+\frac{50\cdot\ 2^{2/3}}{\sqrt[3]{499+3 i \sqrt{111}}}+\sqrt[3]{998+6 i
   \sqrt{111}}\right)\cong 10.331851
$$
and the associated probability eigenvector $\boldsymbol{\nu}=(\nu_1,\nu_2,\nu_3,\nu_4)$ is given by
\begin{equation}\label{n29}
\boldsymbol{\nu}=\left(\dfrac{-3\eta^2+32\eta-9}{4},\dfrac{5\eta^2-54\eta+25}{4},\dfrac{\eta^2-10\eta-3}{4},\dfrac{-3\eta^2+32\eta-9}{4}\right),
\end{equation}
which is, approximately, equal to
$$  \boldsymbol{\nu}\cong (0.344446,\, 0.203947,\, 0.107159,\, 0.344446).
$$
The vector $\boldsymbol{\lambda}=(\lambda_1,\lambda_2,\lambda_3,\lambda_4)=Q\boldsymbol{\nu}$ is given by
\begin{equation}\label{valuel}
\boldsymbol{\lambda}=\left(\dfrac{-3\eta^2+32\eta-9}{4},\dfrac{6\eta^2-64\eta+22}{4},\dfrac{-2\eta^2+22\eta-12}{4},\dfrac{5\eta^2-54\eta+25}{4}\right),
\end{equation}
which is, approximately, equal to
$$  \boldsymbol{\lambda}\cong(0.344446,\, 0.3111078,\, 0.4516059,\, 0.203947).
$$
Notice that $\vert\boldsymbol{\lambda}\vert\cong 1.311107$. 
\begin{proof}[Proof of Lemma \ref{adi}]
Let $I_1,I_2, I_3,I_4$ be the partition of $[0,\vert\boldsymbol{\lambda}\vert]$ defined by
$$ I_1=[x_0,x_1),\quad I_2=[x_1,x_2), \quad I_3=[x_2,x_3), \quad I_4=[x_3,x_4],
$$
where
$$ x_0=0,\quad x_1=\lambda_1,\quad x_2=\lambda_1+\lambda_2,\quad x_3=\lambda_1+\lambda_2+\lambda_3,\quad x_4=\lambda_1+\lambda_2+\lambda_3+\lambda_4=\vert\boldsymbol{\lambda}\vert.
$$
Then $$
\begin{cases} I_1\cong [0,\, 0.344446) \\ I_2\cong [0.344446\,,0.655553) \\ I_3\cong [0.655553,\,1.107159) \\ I_4\cong [1.107159,\,1.311107].
\end{cases}
$$

Let $$ x_0'=0,\quad x_1'=\nu_1, \quad x_2'=\nu_1+\nu_2, \quad x_3'=\nu_1+\nu_2+\nu_3,\quad x_4'=1.$$

By using the equality $\boldsymbol{\lambda}=Q\boldsymbol{\nu}$, by (\ref{formulaT}) and some numerical analysis, we reach Table \ref{tab10}. 
\def\arraystretch{1.3}
\setlength{\tabcolsep}{5pt}
\captionsetup[table]{skip=10pt}
\captionsetup[table]{belowskip=0pt}
\begin{table}[ht!]
\begin{tabular}{ |c|l|l|l|c|}	\hline	
  $i$ & $x_i'$ &  $\big\{T^k(x_i'):0\le k\le N(x_i')-1\big\}$ & $T^{N(x_i')}(x_i')$ & $N(x_i')$ \\  \hline	
   0  & 0 &  $0$ & $0.796052\ldots$ & $1$  \\ \hline
  1 & $\nu_1$ &  $x_1\phantom{aaaaa}$ \quad $1.311107\ldots$ & $0.796052\ldots$  & $2$ \\ \hline
 2 &  $\nu_1+\nu_2$ & $0.548394\ldots$ \quad $x_3$ & $x_4'$ & $2$ \\ \hline
  3 & $\nu_1+\nu_2+\nu_3$ & $x_2$ & $0.451606\ldots$  & $1$ \\\hline
  4 &$1$ &  $1$  & $0.107159\ldots$ & $1$ \\ \hline
 \end{tabular}\caption{}\label{tab10}
\end{table}

 Table \ref{tab10} shows that 
  (H1)-(H2) in Definition \ref{admint} are satisfied for $B=\bigcup_{i=1}^{4} \left\{x_i',T(x'_i),\ldots,T^{N(x'_i)-1}(x'_i)\right\}$ and $a'=x_4'=1$. In fact,
$\mathcal{D}(T)=\{x_1,x_2,x_3\}\subset B$ and $a'\in T(B)$.
Hence, $I'$ is an admissible interval for $T$. By Proposition \ref{277}, for each $1\le i\le 4$, there exist $r_i\ge 1$ and a word $i_0i_1\ldots i_{{r_i}-1}$ over the alphabet $\mathcal{A}=\{1,2,3,4\}$ such that (A1)-(A4) are true. In particular, we have that
$r_i=N(c_i)$, where $c_i=(x_{i-1}'+x_i')/2$. The values of $r_i$ and $i_0 i_1\cdots i_{r_{i}-1}$ are given in Table \ref{tab11}.
\setlength{\tabcolsep}{5pt}
\captionsetup[table]{skip=10pt}
\captionsetup[table]{belowskip=0pt}
\begin{table}[ht!]
\begin{tabular}{ |c|l|l|c|c|c|}	\hline	
  $i$ & $c_i=(x_{i-1}'+x_i')/2$ &  $\big\{T^k(c_i):0\le k\le r_i-1\big\}$ & $T^{r_i}(c_i)$ & $N(c_i)$ & $i_0 i_1\ldots i_{r_i-1}$ \\  \hline	
   1 & $0.172223\ldots$ &  $0.172223\ldots$ & $0.623829\ldots$  & $1$ & $1$ \\ \hline
 2 &  $0.4464201\ldots$ & $0.446420\ldots$ \quad $1.209134\ldots$ & $0.898026\ldots$ & $2$ & $24$ \\ \hline
  3 & $0.601974\ldots$ & $0.601974\ldots$ \quad  $1.053579\ldots$ & $0.053579\ldots$  & $2$ & $23$ \\\hline
  4 &$0.827777\ldots$ &  $0.827777\ldots$  & $0.2793829\ldots$ & $1$ & $3$ \\ \hline
 \end{tabular}\caption{}\label{tab11}
\end{table}
By Corollary \ref{xy17}, Table \ref{tab11} and the equality $\boldsymbol{\lambda}=Q\boldsymbol{\nu}$, we have that the Poincar\'e map $T'$ of $T$ on $I'=[0,1]$ is given by
\begin{equation*}
T'(x)=\begin{cases}
-x +\lambda_1+\lambda_3=-x+\nu_1+\nu_3+\nu_4=-x-\nu_2+1 & \textrm{if}\quad x\in [x_{0}',x_1')\\
\phantom{-} x + \lambda_3 = x + \nu_3 + \nu_4 = x -\nu_1-\nu_2+1 & \textrm{if}\quad x\in [x_{1}',x_2']\\
\phantom{-} x +\lambda_2+\lambda_3-\vert\boldsymbol{\lambda}\vert = x-\lambda_1-\lambda_4 = x-\nu_1-\nu_2 & \textrm{if}\quad x\in (x_{2}',x_3')\\
-x+\lambda_1+\lambda_2+\lambda_3=-x+\nu_1+\nu_2+2\nu_3+\nu_4=-x+\nu_3+1 & \textrm{if}\quad x\in [x_{3}',x_4']
\end{cases}.
\end{equation*}

This concludes the proof of Lemma \ref{adi}.  
\end{proof}
\begin{proof}[Proof of Lemma \ref{rl2}] It suffices to verify the hypotheses of Corollary \ref{48}. By Lemma \ref{adi}, $I'$ is an admissible interval for $T$. Moreover, by the $N(c_i)$-column in Table \ref{tab11} and by the equality  $\boldsymbol{\lambda}=Q\boldsymbol{\nu}$, we reach for $J_i=(x_{i-1}',x_i')$,
$$ \sum_{i=1}^4 r_i \vert J_i\vert=\sum_{i=1}^4 r_i (x_{i}'-x_{i-1}')=\sum_{i=1}^4 r_i\nu_i=\nu_1 +2 \nu_2 + 2 \nu_3+\nu_4=\lambda_1+\lambda_2+\lambda_3+\lambda_4=\vert \boldsymbol{\lambda}\vert,$$
which shows that (H3) is true.
 \end{proof}
  \begin{proof}[Proof of Lemma \ref{aab}]
  The proof consists in verifying the hypotheses (H1)-(H2) in Definition \ref{admint} considering the map $S:[0,1]\to [0,1]$, defined in Subsection \ref{subS}, and the interval $I'=\left[0,\frac{1}{\eta}\right]\cong[0,\,0.096788]$. Notice that $\mathcal{D}(S)=\{y_1,y_2,y_3\}$, where 
$$ y_0=0,\,\, y_1=\nu_1\cong 0.344446, \,\, y_2=\nu_1+\nu_2\cong 0.548394, \,\, y_3=\nu_1+\nu_2+\nu_3=0.655553,\,\, y_4=1.
$$
 Let
 $$ y_0'=0, \quad y_1'=\frac{1}{\eta} y_1,\quad y_2'=\frac{1}{\eta} y_2,\quad y_3'=\frac{1}{\eta} y_3,\quad y_4'=\frac{1}{\eta}.
 $$
 By using the equality $P\boldsymbol{\nu}=\eta\boldsymbol{\nu}$ and some numerical analysis, we reach Table \ref{tab12}. Table \ref{tab12} shows that 
  (H1)-(H2) in Definition \ref{admint} are satisfied for $B=\bigcup_{i=1}^{4} \left\{y_i',S(y'_i),\ldots,S^{N(y'_i)-1}(y'_i)\right\}$ and $a'=y_4'=\frac{1}{\eta}$. In fact,
$\mathcal{D}(T)=\{y_1,y_2,y_3\}\subset B$ and $a'\in S(B)$.
Hence, $I'=\left[0,\frac{1}{\eta}\right]$ is an admissible interval for $S$, which concludes the proof.
\end{proof}
\def\arraystretch{1.3}
\setlength{\tabcolsep}{5pt}
\captionsetup[table]{skip=10pt}
\captionsetup[table]{belowskip=0pt}
\begin{table}[ht!]
\begin{tabular}{ |c|l|l|l|c|}	\hline	
  $i$ & $y_i'$ &  $\big\{S^k(y_i'):0\le k\le N(y_i')-1\big\}$ & $S^{N(y_i')}(y_i')$ & $N(y_i')$ \\  \hline	
   $0$  & $0$ & \begin{minipage}{95mm} ~\\ $0\phantom{.000000\ldots}$ \, $0.796052\ldots$\,\,\, $0.311107\ldots$\,\, $0.484944\ldots$ \\ $0.936550\ldots$ \, $0.170609\ldots$ \,  $0.625442\ldots$ \\ \end{minipage}
  & $0.077048\ldots$ & $7$  \\ \hline
  $1$ & $\dfrac{\nu_1}{\eta}$ &  \begin{minipage}{95mm} ~\\ $0.033338\ldots$ \, $0.762713\ldots$ \, $y_1\phantom{344446\ldots}$  \, $0.796052\ldots$ \\ $0.311107\ldots$ \, $0.484944\ldots$ \, $0.936550\ldots$ \, $0.170609\ldots$ \\ $0.625442\ldots$ \, \\ 
   \end{minipage} & $0.077048\ldots$  & $9$ \\ \hline
 $2$ &  $\dfrac{\nu_1+\nu_2}{\eta}$ & \begin{minipage}{95mm} ~\\ $0.053078\ldots$ \, $0.742974\ldots$ \, $0.364185\ldots$ \, $0.815791\ldots$ \\ $0.291368\ldots$ \, $0.504683\ldots$ \, $0.956289\ldots$ \, $0.150869\ldots$ \\ $0.645182\ldots$ \\\end{minipage} & $y_4'$ & $9$ \\ \hline
  $3$ & $\dfrac{\nu_1+\nu_2+\nu_3}{\eta}$ & \begin{minipage}{95mm} ~\\ $0.063449\ldots$\, $0.732602\ldots$ \, $0.374557\ldots$ \, $0.826163\ldots$ \\ $0.280996\ldots$ \, $0.515055\ldots$ \, $0.966661\ldots$ \, $0.140498\ldots$ \\ $y_3\phantom{655553\ldots}$ \, $0.451605\ldots$ \, $0.903211\ldots$\,
 $0.203947$ \\ $0.592104\ldots$ \\
   \end{minipage} & $0.043710\ldots$  & $13$ \\\hline
  $4$ &$\dfrac{1}{\eta}$ &  \begin{minipage}{95mm} ~\\ $0.096788\ldots$ \, $0.699263\ldots$ \, $0.407895\ldots$ \, $0.859501\ldots$ \\ $0.247658\ldots$ \, $y_2\phantom{599394\ldots}$ \, $1\phantom{000000\ldots}$ \,  $0.107159\ldots$ \\ $0.688892\ldots$ \, $0.418267\ldots$ \, $0.869873\ldots$ \,
$0.237286\ldots$ \\ $0.558765\ldots$ \\
   \end{minipage}  &   $0.010371\ldots$ & $13$ \\ \hline
 \end{tabular}\caption{}\label{tab12}
\end{table}
\begin{proof}[Proof of Lemma \ref{aac}] By Lemma \ref{aab} and Proposition \ref{277}, for each $1\le i\le 4$, there exist $r_i\ge 1$ and a word $i_0i_1\ldots i_{{r_i}-1}$ over the alphabet $\mathcal{A}=\{1,2,3,4\}$ such that (A1)-(A4) are true. In particular, we have that
$r_i=N(c_i)$, where $c_i=(y_{i-1}'+y_i')/2$. The iterates $S^k(c_i)$  are shown in Table \ref{tab13}. 
 The values of $r_i$ and $i_0 i_1\cdots i_{r_{i}-1}$ are given in Table \ref{tab14}. By Corollary \ref{xy17}, Table \ref{tab14} and the equality $\boldsymbol{\lambda}=Q\boldsymbol{\nu}$, we have that
\setlength{\tabcolsep}{5pt}
\captionsetup[table]{skip=10pt}
\def\arraystretch{2}
\captionsetup[table]{belowskip=0pt}
\begin{table}[ht!]
\begin{tabular}{ |c|l|l|c|c|}	\hline	
  $i$ & $c_i=\dfrac{(y_{i-1}'+y_i')}{2}$ &  $\big\{T^k(c_i):0\le k\le r_i-1\big\}$ & $T^{r_i}(c_i)$ & $N(c_i)$  \\  \hline	
   1 & $0.016669\ldots$ &  \begin{minipage}{95mm} ~\\ $0.016669\ldots$\, $0.779382\ldots$ \, $0.327776\ldots$ \, $0.468275\ldots$ \\ $0.919881\ldots$ \, $0.187278\ldots$ \, $0.608773\ldots$ \\ \end{minipage} & $0.060379\ldots$  & $7$   \\ \hline
 2 &  $0.043208\ldots$ &  \begin{minipage}{95mm} ~\\ $0.043208\ldots$ \, $0.752843\ldots$ \, $0.354315\ldots$ \, $0.805921\ldots$ \\ $0.301237\ldots$ \, $0.494814\ldots$ \, $0.946420\ldots$ \,  $0.160739\ldots$ \\ $0.635312\ldots$  \\
   \end{minipage}  & $$0.086918\ldots$$ & $9$   \\ \hline
  3 & $0.0582639\ldots$ & \begin{minipage}{95mm} ~\\ $0.058263\ldots$ \, $0.737788\ldots$ \, $0.369371\ldots$ \,$0.820977\ldots$ \\ $0.286182\ldots$ \, $0.509869\ldots$\, $0.961475\ldots$ \, $0.145684\ldots$ \\ $0.650368\ldots$ \, $0.101973\ldots$  \, $0.694078\ldots$\, $0.413081\ldots$ \\ $0.864687\ldots$ \, $ 0.242472\ldots$ \, $0.553579\ldots$ \\
   \end{minipage} & $0.005185\ldots$  & $15$   \\\hline
  4 &$0.08011894\ldots$ &  \begin{minipage}{95mm} ~\\ $0.080118\ldots$ \, $0.715933\ldots$ \, $0.391226\ldots$ \, $ 0.842832\ldots$ \\ $0.264327\ldots$ \, $0.531724\ldots$ \, $0.983330\ldots$ \, $0.123829\ldots$ \\ $0.672223\ldots$ \, $0.434936\ldots$  \, $0.886542\ldots$ \,
 $0.220617\ldots$\\ $0.575434\ldots$ \\
   \end{minipage}   & $0.027040\ldots$ & $13$   \\ \hline
 \end{tabular}\caption{}\label{tab13}
\end{table}
the Poincar\'e map $S'$ of $S$ on  $\left[0,\frac{1}{\eta} \right]$ is given by
\begin{equation}\label{S'}
S'(x)=\begin{cases}
-x +2 - 2 \nu_1 - 5 \nu_2 - 2 \nu_3 & \textrm{if}\quad x\in (y_{0}',y_1')\\
\phantom{-}x+2 - 3 \nu_1 - 4 \nu_2 - \nu_3 & \textrm{if}\quad x\in (y_{1}',y_2')\\
\phantom{-}x+ 3 - 5 \nu_1 - 6 \nu_2 - \nu_3  & \textrm{if}\quad x\in (y_{2}',y_3')\\
-x+\nu_3 & \textrm{if}\quad x\in (y_{3}',y_4')
\end{cases}.
\end{equation}
By (\ref{S'}) and by the equality $\frac{1}{\eta}\boldsymbol{\nu}=P^{-1}\boldsymbol{\nu}$, it follows that $S'=L\circ S\circ L^{-1}$ on $I'{\setminus}\{y_0',\ldots,y_4'\}$, proving that $S$ in fact self-similar on $\left[0,\frac{1}{\eta}\right]$. This concludes the proof of Lemma \ref{aac}.
 \def\arraystretch{1.3}
\setlength{\tabcolsep}{5pt}
\captionsetup[table]{skip=10pt}
\captionsetup[table]{belowskip=0pt}
\begin{table}[ht!]
\begin{tabular}{ |c|c|l|}	\hline	
  $i$ & $r_i=N(c_i)$ & $i_0i_1\ldots i_{r_i-1}$ \\  \hline	
  1 &7 & 1\,4\,1\,2\,4\,1\,3   \\ \hline
 2 &  9 & 1\,4\,2\,4\,1\,2\,4\,1\,3   \\ \hline
  3 & 15 & 1\,4\,2\,4\,1\,2\,4\,1\,3\,1\,4\,2\,4\,1\,3  \\\hline
  4 &13 & 1\,4\,2\,4\,1\,2\,4\,1\,4\,2\,4\,1\,3 \\ \hline
 \end{tabular}\caption{}\label{tab14}
\end{table} 
 \end{proof}
 \begin{proof}[Proof of Lemma \ref{Sistt}] It suffices to verify the hypotheses of Corollary \ref{332}. By Lemma \ref{aab}, $\left[0,\frac{1}{\eta}\right]$ is an admissible interval for $S$. By Lemma \ref{aac}, $S$ is self-similar on $\left[0,\frac{1}{\eta}\right]$. Let $p_{ij}$ denote the $i,j$-entry of the matrix $P$. By the $N(c_i)$-column in Table \ref{tab14} and by the equality  $P\boldsymbol{\nu}=\eta\boldsymbol{\nu}$, we reach for $J_i=(y_{i-1}',y_i')$,
$$ \sum_{i=1}^4 r_i \vert J_i\vert=\sum_{i=1}^4 r_i (y_{i}'-y_{i-1}')=\sum_{i=1}^4 r_i\frac{\nu_i}{\eta}=\frac{1}{\eta}(7\nu_1 +9 \nu_2 + 15 \nu_3+ 13\nu_4)=\frac{1}{\eta}\sum_{j=1}^4\sum_{i=1}^4 p_{ij} \nu_i=\frac{1}{\eta}\sum_{i=1}^n \eta\nu_i=1,$$
which shows that (H3) is true.
Applying (\ref{mabc}) to the  third column in Table \ref{tab14} yields $M=P$, where $M$ is the matrix associated with $\left(S,\left[0,\frac{1}{\eta}\right]\right)$. Hence, $M$ is positive and (H4) holds. By Corollary \ref{332}, $S$ is topologically transitive. By Lemma \ref{rl2}, $T$ is topologically transitive.
\end{proof}

\noindent\textbf{Acknowledgments}. F. Fernandes was financed  by the Coordena\c c\~ao de Aperfei\c coamento de Pessoal de N\'ivel Superior - Brasil (CAPES) - Finance Code 001. B. Pires was partially supported by grant {\#}2018/06916-0, S\~ao Paulo Research Foundation (FAPESP) and by the National Council for Scientific and Technological Development (CNPq).

\end{document}